\documentclass[12pt]{amsart}

\usepackage{times}

\usepackage{amssymb}

\usepackage{thmtools}

\usepackage{centernot}

\usepackage{tikz}
\usetikzlibrary{arrows.meta}

\usepackage[shortlabels]{enumitem}

\usepackage[spacing=true,kerning=true,babel=true,tracking=true]{microtype}

\usepackage{float}

\usepackage{hyperref}
\hypersetup{
	pdfstartview={XYZ null null 1.00}, 
	pdfpagemode=UseNone, 
	colorlinks,
	breaklinks, 
	linkcolor=blue,
	urlcolor=blue, 
	anchorcolor=blue,
	citecolor=blue
}

\usepackage[capitalise]{cleveref}
\crefformat{footnote}{#2\footnotemark[#1]#3}

\usepackage{geometry}
\usepackage{comment}

\usepackage{xspace}
\usepackage{ bbold }
\usepackage{graphicx}

\usepackage{etoolbox}

\usepackage{thmtools}
\usepackage{thm-restate}

\makeatletter
\def\subsection{\@startsection{subsection}{2}%
  \z@{.8\linespacing\@plus.7\linespacing}{.5\linespacing}%
  {\normalfont\bfseries}}

\def\@seccntformat#1{%
  \protect\textup{%
    \protect\@secnumfont
    \expandafter\protect\csname format#1\endcsname 
    \csname the#1\endcsname
    \protect\@secnumpunct
  }%
}

\patchcmd{\@startsection}
  {\@afterindenttrue}
  {\@afterindentfalse}
  {}{}
 \makeatother

\makeatletter
\def\l@section{\@tocline{1}{5pt}{0pc}{}{}}
\renewcommand{\tocsection}[3]{%
	\indentlabel{\@ifnotempty{#2}{\makebox[20pt][l]{%
				\ignorespaces#1 #2.\hfill}}}\sc #3\dotfill}

\newdimen{\tocsubsecmarg}
\def\l@subsection{\@tocline{2}{3pt}{0pc}{\tocsubsecmarg}{}}
\renewcommand{\tocsubsection}[3]{%
	\indentlabel{\@ifnotempty{#2}{\makebox[30pt][l]{%
				\ignorespaces#1 #2.\hfill}}}#3\dotfill}
\makeatother
\let\oldtocsubsection=\tocsubsection
\renewcommand{\tocsubsection}[2]{\hspace{3em} \oldtocsubsection{#1}{#2}}

\numberwithin{equation}{section}

\theoremstyle{plain}

\newtheorem{lemma}[equation]{Lemma}

\crefname{prop}{Proposition}{Propositions}
\newtheorem{prop}[equation]{Proposition}

\newtheorem{theorem}[equation]{Theorem}

\crefname{obs}{Observation}{Observations}
\newtheorem{obs}[equation]{Observation}

\crefname{cor}{Corollary}{Corollaries}
\newtheorem{cor}[equation]{Corollary}

\makeatletter
\def\@empty{}
\def\ifemptycredit#1{%
	\def\tmp{#1}%
	\ifx\tmp\@empty%
	\else%
	{~(#1)}%
	\fi%
}
\makeatother

\newenvironment{namedthm*}[2][]{
\medskip\par\noindent \textbf{#2}\ifemptycredit{#1}\textbf{.}\itshape\xspace
}{\medskip}

\theoremstyle{definition}

\crefname{defn}{Definition}{Definitions}
\newtheorem{defn}[equation]{Definition}

\crefname{example}{Example}{Examples}
\newtheorem{example}[equation]{Example}

\crefname{question}{Question}{Question}
\newtheorem{question}[equation]{Question}

\crefname{prob}{Problem}{Problem}

\crefname{hypothesis}{Hypothesis}{Hypotheses}
\newtheorem{hypothesis}[equation]{Hypothesis}

\theoremstyle{remark}

\newtheorem*{acknowledge}{Acknowledgments}

\crefname{remark}{Remark}{Remarks}
\newtheorem{remark}[equation]{Remark}

\crefname{claim}{Claim}{Claims}
\newtheorem{claim}[equation]{Claim}
\newtheorem*{claim*}{Claim}

\crefname{assumption}{Assumption}{Assumptions}

\declaretheoremstyle[
spaceabove=\topsep, 
spacebelow=6pt,
headfont=\normalfont\itshape,
notefont=\normalfont, notebraces={(}{)},
bodyfont=\normalfont,
postheadspace=4pt,
qed=\mbox{\smaller[4]$\boxtimes$}
]{claimproofstyle}
\declaretheorem[name={Proof of Claim}, style=claimproofstyle, unnumbered]{pf}

\let\#\mathbb
\let\@\mathcal

\let\dfn\textbf

\newcommand{\F}{\mathbb{F}}

\newcommand{\N}{\mathbb{N}}

\newcommand{\R}{\mathbb{R}}
\newcommand{\Z}{\mathbb{Z}}

\newcommand{\Symb}{S}
\newcommand{\Words}[1][\Symb]{#1^{<\N}}

\newcommand{\shift}{\sigma}
\newcommand{\Rec}{\mathrm{Rec}_P}

\renewcommand{\mod}{\mathop{\mathrm{mod}}\xspace}

\newcommand{\e}{\varepsilon}

\renewcommand{\phi}{\varphi}

\newcommand{\conc}{}
\newcommand{\imp}{\Rightarrow}

\newcommand{\actson}{\curvearrowright}

\newcommand{\onum}[2][th]{$#2^\text{#1}$}

\newcommand{\ie}{i.e.,\xspace}
\renewcommand{\ae}{a.e.\xspace}
\newcommand{\eg}{e.g.,\xspace}

\newcommand*{\defeq}{\mathrel{\vcenter{\baselineskip0.5ex \lineskiplimit0pt \hbox{\scriptsize.}\hbox{\scriptsize.}}}=}

\newcommand{\dom}{\mathrm{dom}}

\newcommand{\set}[1]{\left\{ #1 \right\}}

\newcommand{\rest}[1]{|_{#1}}
\newcommand{\diam}{\mathrm{diam}}

\newcommand{\eqcomment}[1]{\Big[\text{#1}\Big] \hspace{6pt}}

\usepackage{color}
\newcommand{\red}[1]{{\color{red} #1}}
\definecolor{purple}{RGB}{116,0,159}

\definecolor{vert}{RGB}{7,126,26}

\newcommand{\Cyc}{\mathcal{C}_G}

\newcommand{\bdryaction}{\beta}

\newcommand{\symdif}{\mathbin{\triangle}}

\DeclareMathOperator{\reach}{\rightsquigarrow_{P^tP}}

\DeclareMathOperator{\reaches}{\rightsquigarrow_P}
\DeclareMathOperator{\notreaches}{{\centernot\rightsquigarrow}_P}

\DeclareMathOperator{\dancesto}{\leftrightsquigarrow_{P^T P}}

\let\dfn\textbf

\let\@\mathcal
\let\~\widetilde

\newcommand{\mySearrow}{\mathrel{\text{
    \begin{tikzpicture}[baseline=-0.5ex, scale=0.4]
        \draw[double, double distance=1.5pt, thick, -{Implies}] (0,1) -- (1,0);
    \end{tikzpicture}
}}}

\newcommand{\myNearrow}{\mathrel{\text{
    \begin{tikzpicture}[baseline=-0.5ex, scale=0.4]
        \draw[double, double distance=1.5pt, thick, -{Implies}] (0,0) -- (1,1);
    \end{tikzpicture}
}}}

\title{Bounded chaining in measurable dynamics}

\author{Anush Tserunyan}
\address[Anush Tserunyan]{Mathematics and Statistics, McGill University, Montréal, QC, Canada}
\email{anush.tserunyan@mcgill.ca}

\author{Jenna Zomback}
\address[Jenna Zomback]{Mathematics, Amherst College, Amherst, MA, USA}
\email{jzomback@amherst.edu}

\thanks{A.Ts.\ was supported by NSERC Discovery Grant RGPIN-2020-07120.}



\date{\today}

\begin{document}

\begin{abstract}
We introduce a one-parameter family of notions between double ergodicity and metric ergodicity for measure-class preserving (\ie nonsingular) actions of countable groups on standard probability spaces, providing infinitely many new invariants distinguishing weakly mixing actions. 
We call these properties (essentially) $k$-chaining, for $k \in \N$.
We apply this framework to study boundary actions of free groups of finite rank $r \ge 1$, where the boundary is equipped with a stationary Markov measure. 
We prove that in this context, weak mixing is equivalent to $(2r-1)$-chaining, as well as to strict irreducibility of the transition matrix of the Markov measure.
\end{abstract}

\maketitle

\tableofcontents

\section{Introduction}

Ergodicity and weak mixing are central notions in measurable dynamics.
A Borel action $\Gamma\actson (X,\mu)$ of a countable group $\Gamma$ on a standard probability space $(X,\mu)$ is \dfn{ergodic} if the only measurable, $\Gamma$-invariant subsets of $X$ are null or conull. 
Equivalently, for all positive measure sets $A, B \subseteq X$, there is $\gamma \in \Gamma$ such that the intersection $\gamma A \cap B$ has positive measure.

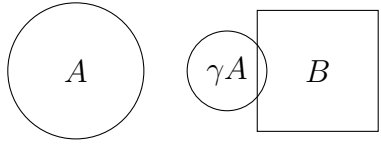
\begin{figure}[ht]\label{fig:ergodic}
\begin{center}
\begin{tikzpicture}

\node[circle, draw, minimum size=1.8cm] (A) at (2,0) {$A$};

\node[circle, draw, minimum size=.8cm] (g3) at (4,0) {$\gamma A$};

\node[rectangle, draw, minimum size=1.6cm] (B) at (5.2,0) {$B$};

\end{tikzpicture}
\end{center}
\end{figure}

An action $\Gamma\actson (X,\mu)$ as above is called \dfn{weakly mixing} if for each ergodic probability measure preserving (pmp) action $\Gamma \actson (Y,\nu)$, the diagonal action $\Gamma \actson (X \times Y, \mu \times \nu)$ is ergodic.

There are a plethora of conditions on group actions that are known to imply weak mixing, including double ergodicity (DE), strong almost transitivity (SAT), metric ergodicity (ME), and unitary ergodicity (UE) (see \cref{subsec:erg_wm_beyond} for definitions). 
The following implications hold for measure-class preserving (mcp) actions of countable groups and are due to Glasner and Weiss \cite[Theorem 1.2]{Glasner-Weiss:weak_mixing}:
\begin{equation*}\label{Glasner-Weiss_diagram}
\begin{matrix}
\text{DE} \\
& \mySearrow \\
& & \text{ME} \implies \text{UE} \iff \text{WM}. \\
& \myNearrow \\
\text{SAT}
\end{matrix}
\end{equation*}
Furthermore, weak mixing trivially implies double ergodicity if the action is pmp, so for pmp actions the above diagram collapses (SAT actions are by definition not pmp).

The main contribution of this paper is the definition and study of a new one-parameter family of notions which provides an infinite stratification of the above hierarchy.
We say that an action $\Gamma \actson (X,\mu)$ as above is \dfn{$k$-chaining} ($k$-C) if for all positive measure sets $A,B \subseteq X$, there are $\gamma_1,\dots,\gamma_k \in \Gamma$ such that the intersections $A \cap \gamma_1 A$, $\gamma_i A \cap \gamma_{i+1} A$ for all $i<k$, and $\gamma_{k} A \cap B$ all have positive measure.
Such a sequence $(A, \gamma_1 A, \gamma_2 A, \dots, \gamma_k A)$ of translates of $A$ is called a \dfn{$k$-chain} from $A$ to $B$.

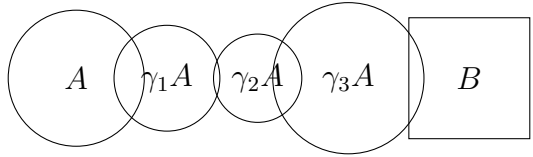
\begin{figure}[ht]\label{fig:bounded_chaining}

\begin{center}
\begin{tikzpicture}

\node[circle, draw, minimum size=1.8cm] (A) at (0,0) {$A$};

\node[circle, draw, minimum size=1.4cm] (g1) at (1.2,0) {$\gamma_1 A$};
\node[circle, draw, minimum size=.6cm] (g2) at (2.4,0) {$\gamma_2 A$};
\node[circle, draw, minimum size=2cm] (g3) at (3.6,0) {$\gamma_3 A$};

\node[rectangle, draw, minimum size=1.6cm] (B) at (5.2,0) {$B$};

\end{tikzpicture}
\end{center}
\end{figure}

We say that the action is \dfn{boundedly chaining} (BC) if it is $k$-chaining for some $k \in \N$. 
Switching the order of quantifiers, we call the action \dfn{chaining} (C) if for all positive measure sets $A,B \subseteq X$, there is a $k$-chain from $A$ to $B$ for some $k \in \N$ (which depends on $A$ and $B$).

As a primary application of this framework, we give a characterization of which stationary Markov measures on the boundaries of free groups make the boundary action weakly mixing, see \cref{intro:wm_boundary}.
The motivation for this is the authors' pointwise ergodic theorem for pmp actions of free groups with averages over trees in the Cayley graph \cite[Theorem 6.2]{TsZ}, which applies exactly to stationary Markov measures which make the boundary action weakly mixing.

\subsubsection*{1-chaining in the literature}
While the notions of bounded chaining and chaining are new (to the best of the authors' awareness), 1-chaining has been studied previously under different names.

In \cite[Page 79]{Fur81}, Furstenberg introduced 1-chaining, although he did not give the property a name.
He proved that for pmp actions of $\Z$, 1-chaining is equivalent to weak mixing.
1-chaining for $\Z$-actions was further developed by Silva and coauthors, see \cite{IKSSW,LS17,HLS}\footnote{In the works of Silva and coauthors, the property of 1-chaining was originally referred to as ``double ergodicity'' and later on renamed ``weak double ergodicity'' to avoid confusion with the already established term double ergodicity, which stands for the ergodicity of the diagonal action on the square.}.

In \cite[Proposition 7.2]{LS17}, Loh and Silva show that in the context of mcp actions of countable groups, 1-chaining implies metric ergodicity.
They also remark that double ergodicity (\ie ergodicity of the diagonal action on the square) implies 1-chaining (see \cref{DE_implies_1C} for a proof of this fact).
Hence, for pmp actions, 1-chaining is equivalent to weak mixing.

\subsubsection*{Our results for mcp actions}

We prove the following implications involving the various degrees of chaining defined above.
Below, the property of \dfn{essential bounded chaining} (ess-BC) is a slight weakening of bounded chaining, see \cref{defn:ess-bounded_chaining}.

\begin{theorem}\label{intro:bounded_chaining_implications}
Let $\Gamma\actson (X,\mu)$ be an mcp Borel action of a countable group $\Gamma$ on a standard probability space $(X , \mu)$.
\begin{enumerate}[(a)]
\item If the action is essentially boundedly chaining (ess-BC), then it is metrically ergodic (ME) (\cref{essential_bounded_chaining_implies_ME}).
\item If the action is weakly mixing (WM), then it is chaining (C) (\cref{cor:weak_mixing_implies_chaining}).
\item If the action is chaining (C), then it is ergodic (E) (\cref{prop:characterization_of_chaining}).
\end{enumerate}
\end{theorem}

We additionally provide counterexamples to several converses of the above implications (see \cref{sec:counterexamples}).
We then have that for all mcp actions,
\begin{align*}
\text{DE} 
\mathop{\substack{
\imp
\\
\red{\nLeftarrow}
}}
\text{1-C}
\imp
\text{2-C}
\imp 
\text{3-C}
\imp
\dots 
\imp
\text{BC}
\mathop{\substack{
\imp
\\
\red{\nLeftarrow}
}}
\text{ess-BC}
\imp
\text{ME}
\imp
\text{WM}
\mathop{\substack{
\imp
\\
\red{\nLeftarrow}
}}
\text{C}
\mathop{\substack{
\imp
\\
\red{\nLeftarrow}
}}
\text{E},
\end{align*}
where the last two converses fail even in the pmp setting.
We also show that the essential bounded chaining hierarchy is strict, providing infinitely many distinct invariants between double ergodicity (DE) and metric ergodicity (ME) in the mcp context:
\begin{align*}
\text{DE} 
\mathop{\substack{
\imp
\\
\red{\nLeftarrow}
}}
\text{1-C}
\mathop{\substack{
\imp
\\
\red{\nLeftarrow}
}}
\text{1-ess-C}
\mathop{\substack{
\imp
\\
\red{\nLeftarrow}
}}
\text{2-ess-C}
\mathop{\substack{
\imp
\\
\red{\nLeftarrow}
}}
\text{3-ess-C}
\mathop{\substack{
\imp
\\
\red{\nLeftarrow}
}}
\dots
\mathop{\substack{
\imp
\\
\red{\nLeftarrow}
}}
\text{ess-BC}
\imp
\text{ME}
\imp
\text{WM}
.
\end{align*}

We also give a characterization of chaining as the actions in which every positive measure set is a complete section (modulo null sets) whose return times generate the group, see \cref{prop:characterization_of_chaining}.
Lastly, we prove that for mcp actions, total ergodicity implies chaining (\cref{prop:TE_implies_chaining}), and the converse is true for mcp actions of $\Z$ (\cref{prop:C_implies_TE_for_Z}).

\subsubsection*{Application: boundary actions of free groups}
Let $r \ge $ denote a natural number, $\F_r$ denote the free group of rank $r$ on the set $\set{a_1,a_2,\dots,a_r}$ of free generators, and put $\Symb \defeq \set{a_1^{\pm 1}, \dots, a_r^{\pm 1}}$.
Let $\partial \F_r$ denote the Gromov boundary of $\F_r$, which we identify with the closed subset of $S^\N$ of infinite reduced words. 
The boundary action $\F_r \actson \partial \F_r$ is then just by concatenation and cancellation.

A well-studied class of probability measures on $\partial\F_r$ is that of Markov measures $\mu$ induced by a positive initial probability distribution $\pi$ on $\Symb$ and a transition matrix $P$ indexed by $\Symb$.
While the boundary action $\F_r \actson (\partial \F_r, \mu)$ is singular (unless we have $P(a,b) = 0$ exactly when $a = b^{-1}$), weak mixing of this action is instrumental for several pointwise ergodic theorems for pmp actions of free groups, particularly, \cite{BN1,BN2} and the authors' ergodic theorem \cite[Theorem 6.2]{TsZ}.
In \cref{intro:wm_boundary} (restated later as \cref{wm_boundary}) we characterize exactly which Markov measures make the boundary action weakly mixing among all stationary Markov chains satisfying a mild symmetry assumption.

Recall that a nonnegative square matrix $P$ indexed by a finite set $\Symb$ is \dfn{irreducible} if for all $a,b\in\Symb$, there exists $n\ge 1$ such that $P^n(a,b)>0$.
Bufetov introduced the stronger notion of \dfn{strict irreducibility}, requiring both $P$ and $P^tP$ to be irreducible \cite{Bufetov:ball_averages,Bufetov:skew_products_erg_thms_Rus,Bufetov:skew_products_erg_thms_Eng}.
For stationary finite-state Markov chains, he proved that strict irreducibility implies ergodicity of every skew extension via a jointly ergodic family of pmp transformations \cite[Theorem 5]{Bufetov:skew_products_erg_thms_Eng}.
Lummerzheim, Pogorzelski, and Zimmermann extended this result to arbitrary state spaces and proved its converse \cite{LPZ}.

To relate skew extensions to boundary actions, suppose that $\Symb$ is the symmetric generating set of $\F_r$ as above and that $P(a,b)=0$ whenever $a=b^{-1}$, so the corresponding Markov measure $\mu$ is supported on $\partial\F_r$.
For an ergodic pmp action $\F_r\actson^\alpha(Y,\nu)$, the skew extension of the shift $\shift$ on $(\partial \F_r, \mu)$ associated with the family of transformations $T_s(y)=s^{-1}\cdot y$ on $Y$ is
\[
\shift_\alpha(x,y)
=(\shift x,T_{x_0}y)
=(x_0^{-1}\cdot x,x_0^{-1}\cdot y).
\]
Then $\shift_\alpha$ generates the same equivalence relation as the diagonal action $\F_r \actson \partial\F_r\times Y$.
Thus, Bufetov's theorem yields the implication \labelcref{intro:item:strict_irred}$\imp$\labelcref{intro:item:wm} in \cref{intro:wm_boundary} below; we give a different proof via $(2r-1)$-chaining in \cref{wm_boundary}.

Furthermore, our converse implication \labelcref{intro:item:wm}$\imp$\labelcref{intro:item:strict_irred} in \cref{intro:wm_boundary} is new and does not follow from the converse in \cite{LPZ}.
Indeed, weak mixing tests only skew extensions arising from jointly ergodic families of invertible pmp transformations satisfying $T_{s^{-1}}=T_s^{-1}$.
When strict irreducibility fails, \cite{LPZ} witnesses nonergodicity using a simple family of \emph{noninvertible} pmp transformations.
We instead construct an ergodic pmp $\F_r$-action whose diagonal product with the boundary action is nonergodic, using a substantially more involved combinatorial argument; see \cref{subsubsec:wm_imp_strict-irred}.

\begin{theorem}[\cref{wm_boundary}]\label{intro:wm_boundary}
Let $\F_r$ be the free group on the set $\set{a_1,a_2,\dots,a_r}$ of free generators, and put $\Symb \defeq \set{a_1^{\pm 1}, \dots, a_r^{\pm 1}}$.
Let $\mu$ be a Markov measure on $\partial \F_r$, whose stochastic matrix $P$ is weakly inverse-symmetric (\cref{symmetry}) and admits a stationary distribution.
Then the following are equivalent:

\smallskip

\begin{enumerate}[(1),leftmargin=*]
\item\label{intro:item:wm} The boundary action $\F_r \actson (\partial \F_r, \mu)$ is weakly mixing.

\item The boundary action $\F_r \actson (\partial \F_r, \mu)$ is boundedly chaining.

\item The boundary action $\F_r \actson (\partial \F_r, \mu)$ is $(2r-1)$-chaining.

\item\label{intro:item:strict_irred} $P$ is strictly irreducible.
\end{enumerate}
\end{theorem}

\subsubsection*{Open questions}
Many possible implications in the hierarchy of notions above weak mixing remain open, notably whether metric ergodicity and unitary ergodicity are equivalent (\cite[Problem 1.4]{Glasner-Weiss:weak_mixing}).
Relatedly, it follows from \cref{ess_1C_not_BC} that metric ergodicity does not imply bounded chaining, but the following is still open:

\begin{question}\label{q:ME_implies_ess-bdd-C}
Is there an mcp action of a countable group $\Gamma$ that is metrically ergodic but not essentially boundedly chaining?
\end{question}

We show in \cref{strict_essBC_hierarchy} that for all $k \in \N$, there is a singular action that is essentially $(k+1)$-chaining but not essentially $k$-chaining.
However, we do not yet have such examples of nonsingular actions.

\begin{question}\label{q:bbd-C_implies_1-C}
Is there an mcp $\Gamma$-action (for some countable group $\Gamma$) that is $(k+1)$-chaining but not $k$-chaining, or at least boundedly chaining but not 1-chaining?
\end{question}

In \cite[Theorem 3.1]{HLS}, the authors provide an example of an infinite measure-preserving free action of $\Z$ which is metrically ergodic but not 1-chaining.
Refining this phenomenon, we provide an example (\cref{ess_1C_not_BC}) of a free mcp action of $\F_2$ which is metrically ergodic, in fact essentially 1-chaining, but not even boundedly chaining.

\begin{question}\label{q:amenable-WM_implies_bdd-C}
Is there an mcp weakly mixing action of $\Z$ (or any other amenable group) that is not boundedly chaining? 
In particular, is the $\Z$-action in \cite[Theorem 3.1]{HLS} boundedly chaining?
\end{question}

\subsubsection*{Organization}
In \cref{sec:prelims}, we give the necessary notation and definitions that are used throughout the paper, we discuss how to ``nonsingularize'' a measure so that we may later work with singular actions, and we provide translations of dynamical properties of Markov measures into combinatorial properties of the associated graph on the state space.
In \cref{sec:def_of_chaining} we introduce the notions of chaining and bounded chaining, which are then systematically studied in \cref{sec:chaining_vs_others} and \cref{sec:bounded_chaining_vs_others}, respectively.
The results in \cref{intro:bounded_chaining_implications} are proved in these two sections.
In \cref{sec:boundary_actions} we prove \cref{intro:wm_boundary} and we show that the essential bounded chaining hierarchy is strict in the mcp setting, using nonsingularizations of boundary actions of free groups as our counterexamples.

\subsubsection*{Declaration on the use of generative AI}
The authors declare that generative AI tools were not used in the research or writing of this manuscript.
The only uses of AI were to generate the TikZ diagrams and to proofread this declaration.

\begin{acknowledge}
The authors are grateful to Cesar Silva for useful discussions regarding the history of 1-chaining.
We thank Elias Zimmermann for many helpful conversations about this project, in particular for pointing out \cref{obs:irreducible} and helping to clarify the relationship between diagonal products of boundary actions and skew extensions.
Thanks to Matt Bowen for providing the proof of \cref{prop:generating_section_pmp}.
We also thank Josh Frisch and Felix Pogorzelski stimulating discussions.
Finally, the authors thank each other so that their acknowledgment includes women.
\end{acknowledge}

\section{Preliminaries}\label{sec:prelims}

Our set $\N$ of natural numbers includes $0$, and we denote $\N^+ \defeq \N \setminus \set{0}$.
For $\e > 0$ and $a,b \in \R$, we write $a \approx_\e b$ to mean $|a-b| \le \e$.

Given a probability space $(\partial \F_r,\mu)$, we write $A=_\mu B$ to mean $\mu(A\triangle B)=0$. 
Similarly, we write $A\subseteq_\mu B$ to mean $\mu(A \setminus B)=0$, and $A >_\mu 0$ to mean $\mu(A) > 0$.

For a set $S$, we denote by $\Words$ the set of finite sequences in $S$, which we refer to as \textbf{words} in $S$.

\subsection{Ergodicity, weak mixing, and beyond}
\label{subsec:erg_wm_beyond}

Throughout, let $\Gamma$ denote a countable group and $(X,\mu)$ denote a standard probability space.

\begin{defn} 
We say that a Borel action $\Gamma \actson^\alpha (X,\mu)$ is
\begin{enumerate}[(a), leftmargin=*]

\item \dfn{probability measure preserving} (\dfn{pmp}) if for all Borel sets $A \subseteq X$ and $\gamma 
\in \Gamma$, $\mu(\gamma A) = \mu(A)$.

\item \dfn{measure class preserving} (\dfn{mcp} or \dfn{nonsingular}) if for all Borel sets $A \subseteq X$ and $\gamma \in \Gamma$, $A$ is null if and only if $\gamma A$ is null.

\item \dfn{singular} if $\alpha$ is not mcp.

\end{enumerate}
\end{defn}

We now recall ergodicity and weak mixing, as well as several other related properties, which were introduced and studied in \cite{Monod:continuous_bdd_cohomology,Burger-Monod} and \cite{Glasner-Weiss:weak_mixing}; see \cite[Remark 0.2]{Glasner-Weiss:weak_mixing} for further references.

\begin{defn}\label{defn:weak_mixing_and_related}
We say that a Borel action $\Gamma \actson (X,\mu)$ is

\begin{enumerate}[(a), leftmargin=*]
\setlength{\itemsep}{4pt}
\item \dfn{ergodic} (E) if every invariant $\mu$-measurable set is null or conull.

\item \dfn{totally ergodic} (TE) if every infinite subgroup of $\Gamma$ acts ergodically.

\item \dfn{weakly mixing} (WM) if for any ergodic pmp action $\beta$ of $\Gamma$ on a standard probability space $(Y,\nu)$, the diagonal action $\Gamma \actson (X \times Y, \mu \times \nu)$ is ergodic.

\item \dfn{metrically ergodic} (ME), also called ergodic with isometric coefficients in \cite{Glasner-Weiss:weak_mixing}, if for every action of $\Gamma \actson Z$ by isometries on a separable metric space $(Z,d)$, every $\mu$-measurable $\Gamma$-equivariant map $\phi: X \to Z$ is $\mu$-essentially constant.

\item \dfn{unitarily ergodic} (UE), also called ergodic with unitary coefficients in \cite{Glasner-Weiss:weak_mixing}, if for every action of $\Gamma \actson H$ by unitary operators on a separable Hilbert space $H$, every $\mu$-measurable $\Gamma$-equivariant map $\phi: X \to H$ is $\mu$-essentially constant.

\item \dfn{doubly ergodic} (DE) if the diagonal action $\Gamma \actson (X^2,\mu^2)$ is ergodic.

\item \dfn{strongly almost transitive} (SAT) if for every positive measure set $A \subseteq X$ and every $\e > 0$, there is $\gamma \in \Gamma$ with $\mu(\gamma^{-1} A) \ge 1 -\e$.
\end{enumerate}
\end{defn}

The diagram below illustrates the implications between these notions in the context of mcp actions, where the nontrivial implications are due to Glasner and Weiss in \cite{Glasner-Weiss:weak_mixing}:
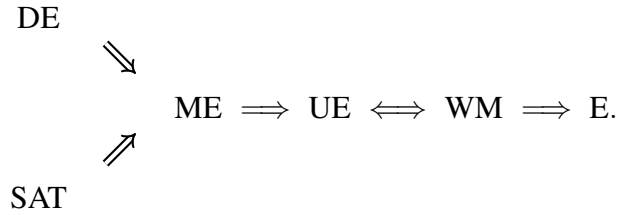
\begin{equation*}
\begin{matrix}
\text{DE} \\
& \mySearrow \\
& & \text{ME} \implies \text{UE} \iff \text{WM} \implies \text{E}. \\
& \myNearrow \\
\text{SAT}
\end{matrix}
\end{equation*}


\subsection{Nonsingularization}

Let $\Gamma \actson (X,\mu)$ be a Borel action of a countable group $\Gamma$ on a standard probability space $(X,\mu)$.
If this action is mcp then each $\gamma \in \Gamma$ maps $\mu$-measurable sets to $\mu$-measurable sets since they differ from a Borel set by a $\mu$-null set.
Conversely, if this action is singular, then there is a $\mu$-null set which is mapped to a $\mu$-nonmeasurable set, so one typically avoids working with singular actions.

However, a natural class of actions studied in this paper are boundary actions of free groups equipped with Markov measures, and unfortunately, many of the interesting ones are singular when the transition matrix has zeros (beyond transitions from a generator to its inverse).
To make these actions easier to work with, we replace the underlying measure with a  nonsingular measure as follows.

\begin{defn}\label{defn:nonsingularization}
Let $\Gamma \actson^\alpha (X,\mu)$ be a Borel action.
Call a measure $\~\mu$ on $X$ a \dfn{nonsingularization} of $\mu$ if the following conditions hold:
\begin{enumerate}[(i)]

\item $\Gamma \actson^\alpha (X,\tilde{\mu})$ is nonsingular;

\item $\mu \ll \~\mu$.
In particular, all $\tilde{\mu}$-measurable sets are $\mu$-measurable;

\item \label{item:same_invariant_null_sets} $\~\mu$ and $\mu$ have the same $\alpha$-invariant null (and hence also conull) sets.

\end{enumerate}

\noindent For example, $\~\mu \defeq \sum_{n \ge 1} 2^{-n} (\gamma_n)_* \mu$ (for any enumeration $\Gamma = (\gamma_n)_{n \ge 1}$) is a nonsingularization.
\end{defn}

We shall verify that nonsingularization preserves fundamental dynamical properties such as ergodicity and weak mixing, but not further strengthenings of weak mixing.
The conservation of ergodicity under nonsingularization is immediate from the condition \labelcref{defn:nonsingularization}\labelcref{item:same_invariant_null_sets} that the original measure and its nonsingularization have the same invariant null and conull sets.

\begin{obs}\label{ergodic_nonsingularization}
Let $\Gamma \actson^\alpha (X,\mu)$ be a Borel action and $\tilde{\mu}$ be a nonsingularization of $\mu$.
Then $\alpha$ is ergodic with respect to $\mu$ if and only if it is ergodic with respect to $\tilde{\mu}$.
\end{obs}

For the conservation of weak mixing under nonsingularization, we need the following.

\begin{prop}\label{product_of_nonsingularization}
Fix a Borel action $\Gamma \actson^\alpha (X,\mu)$ and an mcp action $\Gamma \actson^\beta (Y,\nu)$.
Then for any nonsingularization $\tilde{\mu}$ of $\mu$, the product $\tilde{\mu}\times \nu$ is a nonsingularization of $\mu \times \nu$ with respect to the diagonal action $\Gamma \actson^{\alpha \times \beta} X\times Y$.
\end{prop}

\begin{proof}
The first two conditions follow easily from the definition of the product measure.
To verify the third condition, assume $A \subseteq X\times Y$ is $\tilde{\mu}\times\nu$-measurable and $\alpha\times \beta$-invariant.
Then by Fubini,

\begin{align*}
\mu \times \nu(A) = 0
&\iff 
\nu(A_x)=0 \text{ for $\mu$-a.e.\ } x \in X
\\
\eqcomment{$\beta$ is mcp}
&\iff 
\nu(\bigcup_{\gamma \in \Gamma} \gamma A_x)=0 \text{ for $\mu$-a.e.\ } x \in X
\\
(*)
&\iff 
\nu(\bigcup_{\gamma \in \Gamma} \gamma A_x)=0 \text{ for $\~\mu$-a.e.\ } x \in X
\\
\eqcomment{$\beta$ is mcp}
&\iff 
\nu(A_x)=0 \text{ for $\~\mu$-a.e.\ } x \in X
\\
\eqcomment{Fubini}
&\iff 
\~\mu \times \nu (A) = 0,
\end{align*}
where $(*)$ follows because the set $\set{x \in X: \nu(\bigcup_{\gamma \in \Gamma} \gamma A_x) = 0}$ is $\alpha$-invariant since $\gamma A_x = \gamma (\gamma^{-1} A)_x = A_{\gamma x}$ by the $\alpha \times \beta$-invariance of $A$.
\end{proof}

\begin{remark}\label{product_nonsingularization_counterexample}
The assumption that $\beta$ is mcp in \cref{product_of_nonsingularization} is necessary, as illustrated by the following simple example.
Let $\Z/2\Z$ act on $X = \set{0,1}$ by swapping $0$ and $1$.
This is singular with respect to the Dirac measures $\delta_0$ and $\delta_1$, and the uniform probability measure $\~\mu$ on $X$ is a nonsigulartization of both $\delta_0$ and $\delta_1$.
However, $\~\mu \times \~\mu$ is not a nonsingularization of $\delta_0 \times \delta_1$ because the set $\set{(0,0), (1,1)}$ is $\alpha \times \alpha$-invariant and $\delta_0 \times \delta_1$-null but not $\~\mu \times \~\mu$-null.
\end{remark}

\cref{product_of_nonsingularization,ergodic_nonsingularization} immediately yield:

\begin{cor}\label{WM_nonsingularization}
Let $\Gamma \actson^\alpha (X,\mu)$ be a Borel action and $\tilde{\mu}$ be a nonsingularization of $\mu$.
Then $\alpha$ is weakly mixing with respect to $\mu$ if and only if it is weakly mixing with respect to $\tilde{\mu}$.
\end{cor}

\begin{remark}
The example in \cref{product_nonsingularization_counterexample} shows that nonsingularization preserves neither double ergodicity nor strong almost transitivity. 
Indeed, the action $\alpha$ in the example is doubly ergodic (hence also metrically ergodic) and strongly almost transitive with respect to $\delta_0$, but is not even weakly mixing with respect to $\~\mu$.
\end{remark}

\subsection{Stochastic matrices, associated graphs, and Markov measures}\label{subsec:Markov_chains}

\subsubsection*{Nonnegative matrices and reachability}

Let $\Symb$ be a countable nonempty set of symbols.
For a word $w \in \Words$, we let $|w|$ denote its length and we write $w = (w_0, w_1, \dots, w_{|w|-1})$ where $w_i$ denotes the \onum{i} letter of $w$.

Let $P$ be a nonnegative matrix on the index set $\Symb$. 
In our applications, $P$ consists of transition probabilities between states in $\Symb$, \ie $P$ is a (row) stochastic matrix.

\begin{defn}\label{def:transition_graph}
Let $P$ be a matrix on $\Symb$ with non-negative entries. 
The \dfn{graph associated to $P$} is the directed graph $G_P$ whose vertex set is $\Symb$ and whose edges are exactly the pairs $(a,b)$ with $P(a,b) \ne 0$.
(The following diagram depicts an example.)

\end{defn}


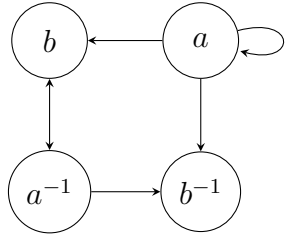
\begin{center}
\begin{tikzpicture}[>=stealth, every node/.style={circle, draw, minimum size=1cm}]

\node (b) at (0,2) {$b$};
\node (a) at (2,2) {$a$};
\node (ainv) at (0,0) {$a^{-1}$};
\node (binv) at (2,0) {$b^{-1}$};

\draw[->] (a) -- (b);
\draw[->] (a) -- (binv);
\draw[<->] (ainv) -- (b);
\draw[->] (ainv) -- (binv);

\draw[->] (a) edge[loop right] (a);
\end{tikzpicture}
\end{center}

For states $a,b \in \Symb$, we say that $b \in \Symb$ is \dfn{$P$-reachable} from $a \in \Symb$ (or that $a$ \dfn{$P$-reaches} $b$), and write $a \reaches b$, if there is a nontrivial (\ie of length at least $1$) directed walk in $G_P$ from $a$ to $b$; 
equivalently, $P^n(a,b) > 0$ for some $n \ge 1$.
$P$ is called \dfn{irreducible} if $a \reaches b$ for all states $a,b \in \Symb$.

A state $s \in \Symb$ is \dfn{$P$-recurrent} if it is $P$-reachable from every state in $\Symb$ (including $s$ itself).
We denote the set of $P$-recurrent states by $\Rec$ and note that $\Rec$ is $P$-invariant, ie if $a \in \Rec$ and $a \reaches b$ then $b \in \Rec$.
In particular, $P$ is irreducible exactly when $\Rec = \Symb$.


\begin{defn}[Bufetov \cite{Bufetov:ball_averages}]
The matrix $P$ is said to be \dfn{strictly irreducible} if both $P$ and $P P^t$ are irreducible.
\end{defn}

\begin{obs}\label{obs:PTP}
Assume $P$ is irreducible. 
Then $P^tP$ is irreducible if and only if $PP^t$ is irreducible. 
\end{obs}

\begin{proof}
Suppose $PP^t$ is irreducible and let $a,b \in \Symb$. 
Since $P$ is irreducible, let $a', b' \in \Symb$ be such that $P (a', a) > 0$ and $P (b', b) > 0$, \ie $P^t (a, a') > 0$ and $P (b', b) > 0$. 
Since $PP^t$ is irreducible, we have $(PP^t)^n (a',b')>0$ for some $n \in \N$.
But then $P^t (P P^t)^n P (a,b) = (P^t P)^{n+1} (a,b)>0$. Hence, $P^t P$ is irreducible.

The proof of the reverse implication is similar.
\end{proof}

Note that for $a,b \in \Symb$, we have that $b$ is $ P^t P$-reachable from $a$ if and only if there is a \dfn{in-and-out zigzag trail} in the graph $G_P$ connecting $a$ and $b$, \ie a finite sequence $a_0, c_0, a_1, c_1 \dots, a_{n-1}, c_{n-1}, a_n \in \Symb$ where $a_0 = a$, $a_n = b$, and $(c_i, a_i), (c_i, a_{i+1})$ are directed edges in $G_P$ for all $i < n$.

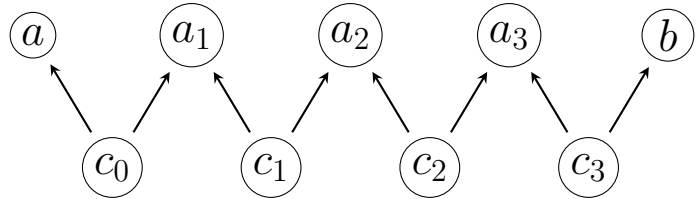
\begin{center}
\begin{tikzpicture}[
    scale=0.7,
    >=stealth,
    every node/.style={
        draw,
        circle,
        minimum size=0.5cm,
        inner sep=2pt,
        font=\Large
    },
    arrow/.style={
        ->,
         thick,
        shorten >=4pt,
        shorten <=4pt
    }
]

\node (a)  at (0,2.5)   {$a$};
\node (a1) at (3,2.5)   {$a_1$};
\node (a2) at (6,2.5)   {$a_2$};
\node (a3) at (9,2.5)   {$a_3$};
\node (b)  at (12,2.5)  {$b$};

\node (c0) at (1.5,0)   {$c_0$};
\node (c1) at (4.5,0)   {$c_1$};
\node (c2) at (7.5,0)   {$c_2$};
\node (c3) at (10.5,0)  {$c_3$};

\draw[arrow] (c0) -- (a);
\draw[arrow] (c0) -- (a1);

\draw[arrow] (c1) -- (a1);
\draw[arrow] (c1) -- (a2);

\draw[arrow] (c2) -- (a2);
\draw[arrow] (c2) -- (a3);

\draw[arrow] (c3) -- (a3);
\draw[arrow] (c3) -- (b);

\end{tikzpicture}
\end{center}

Although we do not use it in our paper, it is worth mentioning that in \cite{LPZ}, the irreducibility of $P^tP$ is shown to be equivalent to a probabilistic condition, namely, that the only $P$-deterministic sets are $\emptyset$ and $\Symb$.
Here a subset $D \subseteq \Symb$ is called \dfn{$P$-deterministic} if $\sum_{d \in D} P(s, d)$ is $0$ or $1$ for each $s \in \Symb$.

\subsubsection*{Markov measures}
Now suppose that $P$ is a (row) stochastic matrix on $\Symb$, and that $\pi$ is a strictly positive probability distribution on $\Symb$.
Then $P$ and $\pi$ induce a Borel probability measure $\mu$ on $\Symb^\N$ defined on the cylinders
$
[w] \defeq \set{w \conc x : x \in \Symb^\N},
$
for $w \in \Symb^n$ and $n \ge 1$, by 
\[
\mu([w]) \defeq \pi(w_0) P(w_0, w_1) \dots P(w_{n-2}, w_{n-1}).
\]
We denote this measure $\mu(P,\pi)$ and call it the \dfn{Markov measure on $\Symb^\N$ induced by $(P,\pi)$}.
We also simply refer to such measures on $\Symb^\N$ as \dfn{Markov measures}.
Many dynamical properties of Markov measures are encoded by the graph associated to the transition matrix given in \cref{def:transition_graph}.

For a non-negative matrix $P$, we call a finite or infinite word $w \defeq (w_0, w_1, w_2, \dots) \in \Words \cup \Symb^\N$ \dfn{$P$-legal} if $w_0, w_1, w_2, \dots$ is a directed walk in the graph $G_P$.
If $\mu$ is a Markov measure on $\Symb^\N$ whose transition matrix is $P$, then we also call $P$-legal words \dfn{$\mu$-legal}.
For a $\mu$-legal word $w \in \Words$, we let
\[
\mu_w \defeq \frac{1}{\mu([w])} \mu \rest{[w]}
\]
denote the normalized restriction of $\mu$ to the cylinder $[w]$.

We record the the main property of Markov measures (see, e.g., \cite[Theorem 5.2.3]{Dur}).

\begin{prop}[Markov property]\label{MP}
Let $\mu$ be a Markov measure on $\Symb^\N$. 
Then for any measurable set $A \subseteq \Symb^\N$ and $\mu$-legal word $w \defeq (w_0, w_1, \dots, w_n) \in \Symb^{n+1}$, we have
\[
\mu_w(\shift^{-n}(A)) = \mu_{w_n}(A),
\]
where $\sigma : \Symb^\N \to \Symb^\N$ is the shift map, \ie $(x_n)_{n \in \N} \mapsto (x_{n+1})_{n \in \N}$.

\end{prop}

For each $s \in \Symb$, let $W_s$ denote the set of finite words in $\Symb$ ending with $s$.
Below let $\mu$ be the Markov measure induced by a transition matrix $P$ and a positive initial distribution $\pi$ on $S$.

\begin{lemma}[Lebesgue density]\label{Lebesgue_density} 
Let $s \in \Rec$.
Then the set $[W_s] \defeq\bigcup_{w \in W_s}[w]$ is $\mu$-conull.
Furthermore, for every positive $\mu$-measure $A \subseteq \Symb^\N$ and $\e > 0$, there is a $\mu$-legal word $w \in \Words$ ending with $s$ and such that $\mu_{w}(A) > 1 - \e$.
\end{lemma}

\begin{proof}
Since $s$ is $P$-recurrent, it must appear infinitely many times in $\mu$-\ae word in $\Symb^\N$, which implies that $[W_s]$ is conull.

For the furthermore part, the regularity of $\mu$ gives a $\mu$-legal word $v \in \Words$ such that $\mu_v(A) > 1 - \e$.
By the previous part, after discarding a conull set, $[v]$ is a union of cylinders $[w]$, where $w$ ranges over all $\mu$-legal words ending with $s$.
We can replace this union with a \emph{disjoint} union over the set of $\mu$-legal words of minimal length which end with $s$, so one such word $w$ must have the property that $\mu_{w}(A) > 1 - \e$.
\end{proof}

\subsection{Existence of positive stationary distributions via reachability}\label{subsec:pos_stationary_distrib}

For a stochastic matrix $P$, a probability vector $\pi$ on $\Symb$ is called \textbf{$P$-stationary} (or a \textbf{$P$-stationary distribution}) if $\pi P = \pi$.
In the Markov chain literature, one typically assumes that the stochastic matrix $P$ on $\Symb$ admits a \emph{positive} stationary distribution.

This property of $P$ can be detected measure-theoretically from any Markov measure $\mu$ whose transition matrix is $P$, even when the initial distribution of $\mu$ is not $P$-stationary.
Indeed, one can show that this property is equivalent to conservativity of the shift map $\shift : (\Symb^\N, \mu) \to (\Symb^\N, \mu)$, \ie the property that every wandering measurable set is $\mu$-null.

We now record that existence of a positive $P$-stationary distribution can also be detected from the graph $G_P$, which is how we will use it.

\begin{prop}\label{stat_dist<=>equivalence}
A stochastic matrix $P$ on a countable set $\Symb$ admits a positive stationary distribution if and only if the reachability relation $\reaches$ on $\Symb$ is an equivalence relation.
\end{prop}

\begin{proof}

If $\reaches$ is an equivalence relation, then for each $\reaches$-equivalence class $C \subseteq \Symb$, the restriction $P \rest{C}$ is irreducible, so by the Perron--Frobenius theorem, there exists a unique positive $P \rest{C}$-stationary distribution $\pi_C$ on $C$.
We obtain a positive probability vector $\pi$ on $\Symb$ by scaling the $\pi_C$s so they add up to a probability vector.
Clearly $\pi$ is $P$-stationary (and non-unique unless there is only one $\reaches$-class).

Conversely, suppose $\reaches$ is not an equivalence relation.
Then we claim it is not symmetric.
Indeed, if it were symmetric, then since for every state $a \in \Symb$, there is a state $b \in \Symb$ that $P(a,b) > 0$, we would also have $b \reaches a$, so $a \reaches a$, since $\reaches$ is always transitive.
Thus $\reaches$ is reflexive, hence an equivalence relation.

Thus, there are (necessarily distinct) states $a,b \in \Symb$ such that $P(a,b) > 0$ but $b \notreaches a$.
Let $B \defeq \set{c \in \Symb : b \reaches c}$, so $B$ is nonempty and $a \notin B$.
Furthermore, the transitivity of $\reaches$ implies that $B$ is $P$-invariant, \ie if $c \in B$ and $P(c,d) > 0$ then $d \in B$.

Now let $\pi$ be a positive distribution on $\Symb$.
Then, intuitively, the mass of $B$, \ie $\pi(B) \defeq \sum_{i \in B} \pi(i)$, stays inside $B$ while a positive part of the mass of $\Symb \setminus B$ flows into $B$, so after one transaction, $B$ will have more mass then before.
Formally, $(\pi P)(B) \ge \pi(B) + \pi(a) \cdot P(a,b) > \pi(B)$.
Thus, $\pi$ is not $P$-stationary.
\end{proof}

\begin{cor}\label{obs:irreducible}
Let $P$ be a stochastic matrix on $S$ which admits a positive stationary distribution. 
If $P^tP$ (or equivalently, $PP^t$) is irreducible, then so is $P$.
\end{cor}

\begin{proof}
Let $a, b \in \Symb$.
Since $P^tP$ is irreducible, there is an in-and-out zigzag trail connecting $a$ to $b$, \ie a finite sequence $a_0, c_0, a_1, c_1, \dots, a_{n-1}, c_{n-1}, a_n \in \Symb$ with $a_0 = a$ and $a_n = b$ such that $(c_i, a_i)$ and $(c_i, a_{i+1})$ are directed edges in $G_P$ for each $i < n$.
By \cref{stat_dist<=>equivalence}, $\reaches$ is an equivalence relation, so we also have $a_i \reaches c_i$ and thus, $a \reaches b$ by transitivity.
Hence $P$ is irreducible.
\end{proof}

Therefore, for a stochastic matrix $P$ which admits a positive stationary distribution, strict irreducibility is equivalent to the irreducibility of $P^t P$.

The following is a well known fact (see, e.g., \cite[Example 6.1.6]{Dur}), but we include an elementary proof here for completeness.

\begin{prop}\label{ergodic<=>irreducible}
Let $P$ be a stochastic matrix on $S$ which admits a positive stationary distribution, and let $\mu$ be any Markov measure on $\Symb^\N$ whose transition matrix is $P$.
Then the shift map $\shift: \Symb^\N \to \Symb^\N$ is $\mu$-ergodic if and only if $P$ is irreducible.
\end{prop}

\begin{proof}
If $P$ is not irreducible, let $a,b \in \Symb$ be such that $b$ is not reachable from $a$.
Then the set $A \subseteq \Symb^\N$ of sequences with only finitely many occurrences of $b$ is $\shift$-invariant and has positive measure since $[a] \subseteq_\mu A$.
On the other hand, by \cref{stat_dist<=>equivalence}, we have $b \reaches b$, so $[b] \subseteq_\mu A^c$. Hence $\mu(A) < 1$, so $\shift$ is not ergodic.

For the converse, suppose that $P$ is irreducible and let $A \subseteq \Symb^\N$ be a shift-invariant set such that both $A$ and $A^c$ have positive measure.
By irreducibility, there exists $s \in \Rec$ (in fact, $\Rec = \Symb$), so by Lebesgue density (\cref{Lebesgue_density}), there are $P$-legal words $u,v \in \Words$ such $u$ and $v$ both end with $s$ and that $\mu_u(A) > 1/2$ and $\mu_v(A^c) > 1/2$.
It then follows from the shift-invariance of $A$ and the Markov property (\cref{MP}) that $\mu_s(A) > 1/2$ and $\mu_s(A^c) > 1/2$, a contradiction.
\end{proof}

\subsection{Boundary actions of free groups}\label{subsec:boundary}

Let $1 \le r \le \infty$ and $\F_r$ denote the free group on the set $\set{a_i}_{0 \le i < r}$ of generators.
Letting $\Symb \defeq \set{a_i^{\pm 1}}_{0 \le < r}$, we let $\partial \F_r$ denote the realization of the Gromov boundary of $\F_r$ as the (closed) subset of $\Symb^\N$ of infinite reduced words.
The natural boundary action $\F_r \actson \partial \F_r$ is then by concatenation and cancellation.
This action is free on a cocountable invariant set and induces the same orbit equivalence relation as the shift map $\shift : \partial \F_r \to \partial \F_r$.

By a Markov measures on the boundary $\partial \F_r$, we mean a Markov measure on $\Symb^\N$ supported on $\partial \F_r$, equivalently, the transition matrix of $\mu$ satisfies $P(a, a^{-1}) = 0$ for all $a \in \Symb$.

Since ergodicity of an action is a property of its orbit equivalence relation, the following is an immediate consequence of \cref{ergodic<=>irreducible}.

\begin{cor}\label{WM_implies_P_irreducible}
Let $\mu$ be a Markov measure on $\partial \F_r$ whose transition matrix $P$ admits a positive stationary distribution.
If the boundary action $\F_r \actson \partial \F_r$ is weakly mixing with respect to $\mu$, then $P$ is irreducible.
\end{cor}


\section{Definitions of chaining and bounded chaining}\label{sec:def_of_chaining}

\begin{defn}\label{defn:chaining}
Let $\Gamma \actson (X, \mu)$ be a Borel group action on a probability space $(X, \mu)$.
Let $k \ge 0$ and $A, B \subseteq X$.
A \dfn{$k$-chain from} $A$ is a sequence $(\gamma_0A, \gamma_1 A, \dots, \gamma_k A)$ of translates of $A$ such that $\gamma_0 = 1_\Gamma$ and $\mu(\gamma_i A \cap \gamma_{i+1}A) > 0$ for each $0 \le i < k$.
We say that the set $A$ \textbf{$k$-chains to} the set $B$ if there is a $k$-chain $(\gamma_0A, \gamma_1 A, \dots, \gamma_k A)$ from $A$ such that $\mu(\gamma_{k}A \cap B) > 0$.
In this case we will also say that $(\gamma_0A, \gamma_1 A, \dots, \gamma_k A)$ is a \dfn{$k$-chain from $A$ to $B$}.

\smallskip

The action $\Gamma \actson (X, \mu)$ is said to be

\begin{enumerate}[(i)]
\item \dfn{$k$-chaining} (\dfn{$k$C}) if $A$ $k$-chains to $B$ for all pairs of sets of positive measure $A,B \subseteq X$.

\item \dfn{boundedly chaining} (\dfn{BC}) if for some $\ell \in \N$, the action is $\ell$-chaining, \ie there exists $\ell \in \N$ such that $A$ $\ell$-chains to $B$ for all pairs of sets of positive measure $A,B \subseteq X$.

\item \dfn{chaining} (\dfn{C}) if for all pairs of sets of positive measure $A,B \subseteq X$, there exists $\ell \in \N$ such that $A$ $\ell$-chains to $B$.
\end{enumerate}
\end{defn}

\begin{remark}\label{remarks:chaining}
\leavevmode
\begin{enumerate}[(a),leftmargin=*,itemsep=6pt]

\item The relation ``$A$ $k$-chains to $B$'' is not symmetric, as demonstrated by the example in \cref{remark:chaining_is_not_symmetric}.

\item The only $0$-chaining actions are trivial actions, \ie those for which $\mu$ is a Dirac measure.

\item Strong mixing implies 1-chaining.

\item\label{SAT_implies_1C} Strong almost transitivity implies 1-chaining.
Indeed, if $A$ and $B$ are positive measure sets, let $\gamma \in \Gamma$ be such that $\mu(\gamma A) > 1 - \min\set{\mu(A), \mu(B)}$, so $\mu(\gamma A \cap A) > 0$ and $\mu(\gamma A \cap B) > 0$.

\item $k$-chaining implies $(k+1)$-chaining for all $k \ge 0$, since there are no requirements that the $\gamma_i$ are pairwise distinct or that $\gamma_i \neq 1_\Gamma$.

\item Bounded chaining and chaining differ by the order of quantifiers.
In bounded chaining, there is a $k \ge 0$ that works uniformly for all sets $A,B$, while in chaining, each pair of sets $A,B$ admits their own $k \ge 0$, which may not be uniformly bounded.
Thus, bounded chaining implies chaining, but the converse easily fails, for example, for any irrational rotation.

\item Nonsingularization does not preserve $k$-chaining for any $k \ge 0$, nor does it preserve chaining.
For example, the (singular) action in \cref{product_nonsingularization_counterexample} is 0-chaining, but every nonsingularization of it is not chaining.
\end{enumerate}
\end{remark}

As we will see in \cref{ess_1C_not_BC}, it is \emph{not} enough to verify $k$-chaining on a generating algebra of measurable sets; indeed, the counterexample we provide is $1$-chaining on cylinder sets but it is not boundedly chaining.
However, the following shows that for pmp actions, a \emph{quantitative version} of $k$-chaining on a generating algebra does imply $k$-chaining on all sets.
Although we do not use this proposition, we include it here for completeness.
We only state and prove this for $k=1$ because by \cref{pmpwmchaining}, bounded chaining is equivalent to 1-chaining for pmp actions.

\begin{prop}\label{1-chaining_on_algebra}
Let $\Gamma\actson^\alpha (X,\mu)$ be a pmp action, and let $\mathcal{A}$ be an algebra of subsets of $X$ generating the Borel sigma-algebra of $X$.
Suppose that there is a continuous, symmetric function $f: \R^{>0}\times \R^{>0} \to \R^{>0}$ such that for all $A,B \in \mathcal{A}$ there is $\gamma \in \Gamma$ with $\mu(\gamma A\cap A) \geq f(\mu(A), \mu(A))$ and $\mu(\gamma A\cap B) \geq f(\mu(A), \mu(B))$.
Then $\alpha$ is 1-chaining.
\end{prop}

\begin{proof}
Let $A$ and $B$ have positive measure.
Let
\[
0 < \e < \frac{1}{2} \min \set{f(\mu(\tilde{A}),\mu(\tilde{A})), f(\mu(\tilde{A}),\mu(\tilde{B}))}
\]
and let $0 < \delta < \e/2$ be small enough so that for any $a \approx_\delta \mu(A)$ and $b \approx_\delta \mu(B)$, we have $f(a,a) \approx_\e f(\mu(A), \mu(A))$ and $f(a,b) \approx_\e f(\mu(A), \mu(B))$.

Approximate $A$ and $B$ with sets $\tilde{A}$ and $\tilde{B}$ in $\mathcal{A}$ so that $\mu(A \symdif \tilde{A})$ and $\mu(B \symdif \tilde{B})$ are less than $\delta$.
Let $\gamma \in \Gamma$ be such that $\mu(\gamma \tilde{A} \cap \tilde{A}) \ge f(\mu(\tilde{A}),\mu(\tilde{A}))$ and $\mu(\gamma \tilde{A} \cap \tilde{B}) \ge f(\mu(\tilde{A}),\mu(\tilde{B}))$.
Then since 
\[
\mu((\gamma A \cap A) \symdif (\gamma \tilde{A} \cap \tilde{A}))
\le 
\mu(A \symdif \tilde{A}) + \mu(\gamma A \symdif \gamma \tilde{A})
=
2\mu(A \symdif \tilde{A})
<
\e,
\]
we have $\mu(\gamma A \cap A) \approx_\e \mu(\gamma \tilde{A} \cap \tilde{A}) \ge f(\mu(\tilde A, \tilde A)) \approx_\e f(\mu(A, A)) > 2 \e$, so $\mu(\gamma A \cap A) > 0$.
Analogously,
$\mu(\gamma A \cap B) > 0$, so $\alpha$ is 1-chaining.
\end{proof}

For example, if $\alpha$ is a Bernoulli shift then $f(x,y) = xy$ satisfies the hypothesis of \cref{1-chaining_on_algebra}.
This yields the fact that any Bernoulli shift is 1-chaining, which was already immediate from the fact that Bernoulli shifts are strongly mixing.

\section{Chaining}\label{sec:chaining_vs_others}

This section is devoted to chaining and its relationship with weak mixing (WM), total ergodicity (TE), and ergodicity (E).
It is clear that for mcp group actions, chaining implies ergodicity, and we characterize chaining as ergodicity plus another condition (\cref{prop:characterization_of_chaining}).
We also prove that chaining is implied by weak mixing (\cref{cor:weak_mixing_implies_chaining}), as well as by total ergodicity under some nontriviality assumptions (\cref{prop:TE_implies_chaining}).

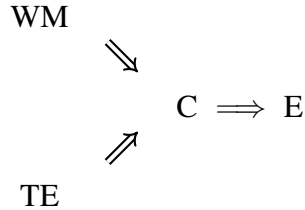
\begin{equation*}\label{Chaining_diagram}
\begin{matrix}
\text{WM} \\
& \mySearrow \\
& & \text{C} \implies \text{E} \\
& \myNearrow \\
\text{TE}
\end{matrix}
\end{equation*}

Even for free pmp actions, weak mixing does not imply total ergodicity: an example of a weak mixing free pmp action of $\F_2$ was given in \cite[Example 4.1]{TD15}. 
In \cref{extension_of_TD}, we give a general class of groups which admit weakly mixing but not totally ergodic free pmp actions. 
That total ergodicity does not imply weak mixing is witnessed by any irrational rotation.
It then follows that chaining implies neither weak mixing nor total ergodicity, and there are easy examples (see \cref{example:ergodic_not_chaining}) of free pmp actions which are ergodic but not chaining.
However, we show that for mcp $\Z$ actions, chaining and total ergodicity are equivalent (\cref{prop:C_implies_TE_for_Z}).

\subsection{Characterization of chaining}

\begin{example}[Ergodicity does not imply chaining]\label{example:ergodic_not_chaining}
Let $\Z/2\Z$ act on $X = \set{0,1}$ by swapping $0$ and $1$, where $X$ has the uniform probability measure.
This is pmp and ergodic but not chaining.

For an example of a free pmp $\Z$ action, let $T: X \to X$ be any free ergodic transformation (for example, an irrational rotation) and let $\~T: X \times \set{0,1} \to X \times \set{0,1}$ be the suspension of $T$, \ie $\~T(x,0) \defeq (x, 1)$ and $\~T(x,1) \defeq (Tx, 0)$. 
This is pmp and ergodic but not chaining.
\end{example}

Thus chaining is stronger than ergodicity and the following definition identifies this additional property.

\begin{defn}\label{def:group_generating}
For a group action $\Gamma \actson (X, \mu)$ on a probability space, we say that a measurable subset $A \subseteq X$ is \dfn{group generating} if the set 
\[
\Gamma_A \defeq \set{\gamma \in \Gamma: \gamma A \cap A >_\mu 0}
\]
of \dfn{return times} of $A$ generates $\Gamma$.
\end{defn}

\begin{lemma}\label{prop:Y_not_conull}
Let $\Gamma\actson (X,\mu)$ be an mcp measurable action.
Let $A$ be a measurable subset of $X$, let $\Delta$ be the subgroup of $\Gamma$ generated by the set of return times $\Gamma_A$, and let $Y \defeq \Delta A$.
For all $\gamma_1, \gamma_2 \in \Gamma$, if $\gamma_1 Y \cap \gamma_2 Y$ has positive measure, then $\gamma_1 \in \gamma_2 \Delta$. 

In particular, $\gamma Y \cap Y$ is null for all $\gamma \in \Gamma \setminus \Delta$.
\end{lemma}

\begin{proof}
If $\gamma_1 Y \cap \gamma_2 Y$ has positive measure, then for some $\delta_1, \delta_2 \in \Delta$, $\gamma_1 \delta_1 A \cap \gamma_2 \delta_2 A >_\mu 0$, so 
$\delta_2^{-1} \gamma_2^{-1} \gamma_1 \delta_1 A \cap A >_\mu 0$,
hence $\delta_2^{-1}\gamma_2^{-1}\gamma_1\delta_1 \in \Gamma_A$.
Therefore $\gamma_2^{-1}\gamma_1 \in \Delta$, so $\gamma_1 \in \gamma_2\Delta$.
\end{proof}

\begin{prop}[Characterization of chaining]\label{prop:characterization_of_chaining}
An mcp action $\Gamma\actson^\alpha (X,\mu)$ of a countable group $\Gamma$ is chaining if and only if it is ergodic and every positive measure set is group generating.
\end{prop}

\begin{proof}
$\implies$: 
Suppose $\alpha$ is not ergodic, so let $A$ be a measurable invariant set with $0 < \mu(A) < 1$. 
Then for all $\gamma \in \Gamma$, $\gamma A \cap A^c = \emptyset$, so $\alpha$ is not chaining.

Now suppose that $A$ is a positive measure set that is not group generating.
Let $\Delta$ be the subgroup of $\Gamma$ generated by $\Gamma_A \defeq \set{\gamma \in \Gamma: \gamma A \cap A >_\mu 0}$, and let $B \defeq X \setminus \Delta A$.
Then by \cref{prop:Y_not_conull}, $B >_\mu 0 $ and $A$ does not $k$-chain to $B$ for any $k \in \N$.
So $\alpha$ is not chaining.

$\impliedby$:
Let $A$ be a positive measure set. 
Since $\alpha$ is ergodic, it's enough to show that for all $\gamma \in \Gamma$, for some $k \in \N$, $A$ $k$-chains to $\gamma A$.
Let $\Delta \defeq \set{\gamma \in \Gamma: \text{ for some } k \in \N, \; A \text{ $k$-chains to } \gamma A}$.
Since $\Delta$ is a generating set for $\Gamma$, we just need to verify that $\Delta$ is a subgroup.

Let $\delta_1, \delta_2 \in \Delta$.
Let $\gamma_1, \dots, \gamma_k$ witness that $A$ $k$-chains to $\delta_1 A$, and let $\gamma_{k+1}, \dots , \gamma_{k+j}$ witness that $A$ $j$-chains to $\delta_2 A$.
Then $\gamma_1, \dots, \gamma_k, \delta_1, \delta_1 \gamma_{k+1}, \dots , \delta_1 \gamma_{k+j}$ witness that $A$ $(k+j)$-chains to $\delta_1\delta_2 A$.
Therefore, $\delta_1\delta_2 \in \Delta$.
\end{proof}

\subsection{Weak mixing implies chaining}

We now show that if an mcp action is weakly mixing, then every positive measure set is group generating (and hence the action is chaining).
Matt Bowen kindly provided the following statement in the pmp context, and we will then generalize this result to the mcp setting (\cref{prop:generating_section}).

\begin{prop}\label{prop:generating_section_pmp}
Let $\Gamma\actson (X,\mu)$ be a pmp weakly mixing action.
Then every positive measure set $A \subseteq X$ is group generating.
\end{prop}

\begin{proof}
Let $A$ be a positive measure set, and let $\Gamma_A$ be the set of return times of $A$. 
By \cref{prop:Y_not_conull}, the cosets of $\Delta$ produce pairwise disjoint (up to measure zero) translates of $A$. 
Since $\alpha$ is pmp and $\mu(A)>0$, $\Delta$ is finite index.

Observe that the coset action $\Gamma\actson \Gamma/\Delta$ is ergodic with respect to any probability measure on the finite set $ \Gamma/\Delta$.
Because $\alpha$ is weakly mixing, the diagonal action $\Gamma\actson X \times \Gamma/\Delta$ is ergodic, where $\Gamma/\Delta$ is equipped with the uniform probability measure.
Finally, since $\bigcup_{\gamma \in \Gamma} (\gamma\Delta A \times \set{\gamma\Delta})$ is a measurable invariant set, we conclude that $\Delta = \Gamma$.
\end{proof}

While the above proof fails in the non-pmp setting, we recover this result for mcp actions using an argument of the same spirit, via the equivalence of weak mixing with unitary ergodicity.

\begin{prop}\label{prop:generating_section}
Let $\Gamma\actson (X,\mu)$ be an mcp weakly mixing action.
Then every positive measure set $A \subseteq X$ is group generating.
\end{prop}

\begin{proof}
Suppose $A \subseteq X$ is a positive measure set that is not group generating.
Let $\Gamma_A$ be the set of return times of $A$, so $\Delta \defeq \langle \Gamma_A \rangle \neq \Gamma$.
Let $Y \defeq \Delta A$, so $Y$ is not conull by \cref{prop:Y_not_conull}.

By \cite[Theorem 1.2]{Glasner-Weiss:weak_mixing}, it suffices to show that $\alpha$ is not unitarily ergodic.
Let $\Gamma \actson \ell_2(\Gamma / \Delta)$ by shift, \ie $(\gamma\cdot x)(\delta \Delta) \defeq x(\gamma^{-1}\delta \Delta)$, so the action is by unitary operators.
Consider the map $\varphi: X \to \ell_2(\Gamma/\Delta)$ defined by
\[
(\varphi x)(\delta \Delta) = 
\begin{cases}
1 \text{ if } x \in \delta Y
\\
0 \text{ otherwise}
\end{cases}
\]
Then $\varphi$ is measurable since $\alpha$ is mcp, $\varphi$ is equivariant by construction, and $\varphi$ is not essentially constant since $Y$ is not conull.
\end{proof}

\cref{prop:generating_section,prop:characterization_of_chaining} immediately yield the following.

\begin{cor}\label{cor:weak_mixing_implies_chaining}
Let $\Gamma\actson (X,\mu)$ be an mcp measurable action.
If the action is weakly mixing, then it is chaining.
\end{cor}

\subsection{Total ergodicity in relation to chaining and weak mixing}

\begin{prop}\label{prop:TE_implies_chaining}
Let $\Gamma \actson (X,\mu)$ be an mcp action of a countably infinite group $\Gamma$ on a standard atomless probability space $(X,\mu)$.
If $\alpha$ is totally ergodic, then $\alpha$ is chaining.
\end{prop}

\begin{proof}
Suppose $\Gamma\actson^\alpha (X,\mu)$ is totally ergodic.
Because $\Gamma$ is infinite, the action $\alpha$ is ergodic.
Thus, by \cref{prop:characterization_of_chaining} it suffices to fix a positive measure set $A \subseteq X$ and show that the subgroup $\Delta \le \Gamma$ generated by the return times to $A$ is equal to $\Gamma$.

\begin{claim*}
$\Delta$ is infinite.
\end{claim*}
\begin{pf}
Suppose towards a contradiction that $\Delta$ is finite.
Then the action of $\Delta$ on $Y \defeq \Delta A$ admits a Borel transversal $Z \subseteq Y$, which necessarily has positive measure.
Since $\mu$ is atomless, we can partition $Z = Z_0 \sqcup Z_1$ into disjoint Borel subsets of positive measure.
By the ergodicity of the action of $\Gamma$, there is $\gamma \in \Gamma$ such that $\gamma Z_0 \cap Z_1$ has positive measure.
In particular, $\mu(\gamma Y \cap Y) > 0$, so $\gamma \in \Delta$ by \cref{prop:Y_not_conull}, which contradicts $Z$ being a transversal for $\Delta \actson Y$.
\end{pf}

By this claim and total ergodicity, $\Delta$ acts ergodically, so $\Delta A$ is conull.
Thus, for each $\gamma \in \Gamma$, there is $\delta \in \Delta$ such that $\delta A \cap \gamma A$ has positive measure.
Therefore, $A \cap \delta^{-1}\gamma A$ has positive measure, so $\delta^{-1} \gamma \in \Delta$, hence $\gamma \in \Delta$.
Thus, $\Delta = \Gamma$.
\end{proof}

In general, chaining does not imply total ergodicity.
As mentioned above, even weak mixing does not imply total ergodicity in general.
On the other hand, it is well known that for $\Gamma = \Z$, weak mixing \textit{does} imply total ergodicity.
We show below that actually, even chaining implies total ergodicity for all mcp $\Z$ actions.

\begin{prop}\label{prop:C_implies_TE_for_Z}
For mcp actions of $\Z$, chaining is equivalent to total ergodicity.
\end{prop}

\begin{proof}
That total ergodicity implies chaining is the content of \cref{prop:TE_implies_chaining}.
Suppose towards a contradiction that an mcp action $\Z \actson (X,\mu)$ is chaining but not totally ergodic, and let $m \in \N$ be such that $T^m$ is not ergodic, where $T(x) \defeq 1 \cdot x$.
Let $A \subseteq X$ be a $T^m$-invariant set with $0 < \mu(A) < 1$.

We will use the following consequence of chaining: for each $B \subseteq A$ and $C \subseteq A^c$ of positive measure, there is a ``chain-link'' which straddles $A$ and $A^c$, \ie $A \cap T^n(B) > 0$ and $A^c \cap T^n(B) > 0$ for some nonzero $n \in \Z$.
Furthermore, $T^m A = A$ implies that $n \ne 0 \; (\mod m)$.

First, let $n_0 \ne 0 \; (\mod m)$ be such that $\mu(A \cap T^{n_0}A) > 0$ and $\mu(A^c \cap T^{n_0}A) > 0$.
Next, let $n_1 \ne 0 \; (\mod m)$ be such that $\mu(A \cap T^{n_1}(A \cap T^{n_0}A)) = \mu(A \cap T^{n_1} A \cap T^{n_1+n_0} A) > 0$ and $\mu(A^c \cap T^{n_1} A \cap T^{n_1+n_0} A) > 0$.
Note that $n_1 + n_0 \ne 0 \; (\mod m)$ since $T^m A = A$, so the numbers $0, n_1, n_1 + n_0$ are distinct $(\mod m)$.

Iterating this, we obtain a sequence $(n_k)_{k \in \N}$ of nonzero ($\mod m$) integers such that for all $k \in \N$ the set $T^{n_k} A \cap T^{n_k+n_{k-1}} A \cap \dots \cap T^{n_k + \dots + n_0} A$ meets both $A$ and $A^c$ in positive measure sets.
It is then straightforward to verify by induction on $k \in \N$ that the tail sums $n_k + n_{k-1} + \cdots + n_i \; (\mod m)$, where $0 \le i \le k$, are nonzero and pairwise distinct.
However, this is impossible for $k \geq m$.
\end{proof}

The last proof heavily relied on the modular arithmetic and 1-dimensionality of $\Z$.
Since total ergodicity is not an upward closed property for groups, while chaining is, it is not surprising that there are weak mixing actions of groups larger than $\Z$ which are not totally ergodic.
Indeed, the $\F_2$-action in \cite[Example 4.1]{TD15} is weakly mixing (and hence chaining by \cref{cor:weak_mixing_implies_chaining}) but not totally ergodic.
We extend this result to a large class of groups.

\begin{prop}\label{extension_of_TD}
Let $\Gamma$ be a countable group which is an extension of an infinite group by an infinite group, \ie there is a surjective homomorphism from $\Gamma$ onto an infinite group with an infinite kernel.
Then there exists a free pmp action of $\Gamma$ that is weakly mixing but not totally ergodic.
In particular, this is true for $\Gamma \defeq \Z^d$ and $\Gamma \defeq \F_d$ with $d \ge 2$.
\end{prop}

\begin{proof}
Let $\Gamma$ be as above, and let $\phi: \Gamma \to \Delta$ be the aforementioned homomorphism.
Let $\Delta \actson (X, \mu)$ be any weakly mixing free pmp action (\eg a Bernoulli shift). 
Extend this to an action $\Gamma \actson^\alpha (X, \mu)$ through $\phi$, so $\ker(\phi)$ acts trivially.
Then by \cref{wm_upward_closed}\labelcref{item:pmp_wm_upward-closed}, since the original $\Delta$-action is weakly mixing, $\alpha$ is weakly mixing.
Observe that $\alpha$ is not totally ergodic since the action of $\ker(\phi)$ is not ergodic.
However, $\alpha$ is not free.

To obtain a free action, let $\Gamma \actson^\beta (Y, \nu)$ be any weakly mixing free pmp action.
Then the diagonal action $\Gamma \actson^{\alpha \times \beta} (X \times Y, \mu \times \nu)$ is still weakly mixing since the diagonal product of two pmp weakly mixing actions is weakly mixing (indeed, if $\rho$ is an ergodic pmp $\Gamma$-action then $\beta \times \rho$ is ergodic pmp since $\beta$ is weakly mixing, hence $\alpha \times \beta \times \rho$ is ergodic since $\alpha$ is weakly mixing).
However, $\alpha \times \beta$ is not totally ergodic, since the action of $\ker(\phi)$ is not ergodic.
\end{proof}

\begin{question}
For which countable groups does chaining (resp., weak mixing) imply total ergodicity for free pmp actions?
Could it be that the converse of \cref{extension_of_TD} holds, \ie the class of countable groups for which weak mixing implies total ergodicity is exactly that of extensions of infinite groups by infinite subgroups?
\end{question}

\section{Bounded Chaining}\label{sec:bounded_chaining_vs_others}

In this section we place bounded chaining in the heirarchy of weak mixing and related conditions for mcp group actions.
It is easy to see (\cref{DE_implies_1C}) that double ergodicity implies 1-chaining, which is explicitly observed in \cite{LS17}, and we prove that bounded chaining implies metric ergodicity (\cref{essential_bounded_chaining_implies_ME}), which in turn implies weak mixing by \cite[Theorem 1.1-1.2]{Glasner-Weiss:weak_mixing}.
In fact we show that a weakening of bounded chaining, which we call \emph{essential bounded chaining} (see \cref{defn:ess-bounded_chaining}), already implies metric ergodicity.
Thus, we have:
\[
\text{DE} \implies
\text{1C} \implies
\text{BC} \implies
\text{ess-BC} \implies
\text{ME}
\implies
\text{UE}
\iff
\text{WM}.
\]  

As previously noted, it is trivial that measure preserving weakly mixing actions are doubly ergodic, so in the pmp setting all of the above properties are equivalent.

\begin{remark}\label{wm_upward_closed}
\leavevmode
\begin{enumerate}[(a), leftmargin=*]
\item\label{item:pmp_wm_upward-closed} It easily follows from the equivalence of weak mixing with double ergodicity for pmp actions that weak mixing is ``upward closed'' for pmp actions.
That is, given $\Delta\le \Gamma$ a pmp action $\Gamma\actson (X, \mu)$, if the restriction of the action to a $\Delta$-action is weakly mixing, then the $\Gamma$-action is also weakly mixing.
Indeed, if the $\Delta$-action is doubly ergodic, then so is the $\Gamma$-action.

\item It is also clear that $k$-chaining is ``upward closed'' (in the above sense) for all mcp actions: the group elements from $\Delta$ that witness a $k$-chain from $A$ to $B$ still witness a $k$-chain for the $\Gamma$-action.
\end{enumerate}  
\end{remark}

\subsection{Double ergodicity implies 1-chaining}

Firstly, we record the following observation, which is also stated in \cite[Section 7.1]{LS17}.

\begin{prop}\label{DE_implies_1C}
Let $\Gamma$ be a countable group and $(X, \mu)$ be a measure space.
If a measurable action $\Gamma \actson (X,\mu)$ is doubly ergodic, then it is 1-chaining.
\end{prop}

\begin{proof}
We prove the contrapositive.
Suppose the action is not 1-chaining, so there exists positive measure sets $A, B \subseteq X$ such that for all $\gamma \in \Gamma$, either $\gamma A \cap A$ or $\gamma  A \cap B$ is null.
Then 
\[
\tilde{A} \defeq \bigcup_{\gamma \in \Gamma} \gamma (A \times A) = \bigcup_{\gamma \in \Gamma} (\gamma A \times \gamma A)
\]
is a $\Gamma$-invariant positive measure subset of $X \times X$ but its intersection with the positive measure set $A \times B$ is null since $\Gamma$ is countable.
\end{proof}

\begin{remark}
By results of Glasner and Weiss in \cite{Glasner-Weiss:weak_mixing}, 1-chaining does not imply double ergodicity.
Indeed, they show that strong almost transitivity does not imply double ergodicity for $\Z$-actions (\cite[Example 3.5]{Glasner-Weiss:weak_mixing}). 
Since strong almost transitivity \emph{does} imply 1-chaining (\cref{remarks:chaining}\labelcref{SAT_implies_1C}), it follows that 1-chaining does not imply double ergodicity even for $\Z$-actions.

Furthermore, in \cite[Section 6]{BFMS}, the authors provide an example of an infinite measure-preserving $\Z$-action that is 1-chaining but not doubly ergodic.
It is not hard to check that an infinite measure-preserving action cannot be strongly almost transitive, even with respect to an equivalent probability measure.
\end{remark}

\subsection{Bounded chaining implies metric ergodicity}

In \cite[Proposition 7.2]{LS17}, Loh and Silva prove that 1-chaining implies metric ergodicity for nonsingular group actions (using different terminology).
Here we will show more generally that bounded chaining implies metric ergodicity;
in fact, even the following weakening of bounded chaining implies metric ergodicity.

\begin{defn}\label{defn:ess-bounded_chaining}
Let $\Gamma\actson (X, \mu)$ be a Borel action of a countable group $\Gamma$ on a standard probability space $(X,\mu)$, and let $k \in \N$.

We say that the action is \dfn{$k$-chaining on a measurable subset $W \subseteq X$} if for all sets $A,B \subseteq W$ of positive measure, $A$ $k$-chains to $B$ (where the $k$-chain is not required to be contained in or even intersect $W$).

We say that the action $\Gamma\actson (X, \mu)$ is \dfn{essentially $k$-chaining} if it is $k$-chaining on $\Gamma$-complete section $W \subseteq X$ (\ie
$\bigcup_{\gamma\in\Gamma} \gamma W =_\mu X$), which is also group generating (see \cref{def:group_generating}).

We also say that the action is \dfn{essentially boundedly chaining} if it is essentially $k$-chaining for some $k \in \N$.
\end{defn}

\begin{remark}
The example in \cref{product_nonsingularization_counterexample}
demonstrates that essential bounded chaining does not pass to a nonsingularization (since the only valid choice of $W$ is $\set{0,1}$ in that example).
\end{remark}

To prove that essential bounded chaining implies metric ergodicity, we need the following fact.

\begin{lemma}\label{ball-separation}
Let $(Z,d)$ be a second countable metric space equipped with a Borel probability measure $\mu$.
If $\mu$ is not a Dirac measure, then for each $k \ge 0$ there are balls $B_0, B_1 \subseteq X$ of arbitrarily small diameter and of positive measure such that $d(B_0,B_1) > k \cdot \max\set{\diam(B_0), \diam(B_1)}$.
\end{lemma}
\begin{proof}
The second countability of $\Z$ ensures that the support of $\mu$ (\ie the $\subseteq$-least $\mu$-conull closed set) is well-defined.
The measure $\mu$ not being Dirac  implies that the support of $\mu$ contains at least two distinct points $x_0$ and $x_1$.
Let $B_0$ and $B_1$ be balls of positive diameter $\delta < \frac{d(x_0, x_1)}{(k+1)}$ centered at $x_0$ and $x_1$, respectively, so $d(B_0, B_1) \ge d(x_0, x_1) - \delta > (k + 1 - 1) \delta = k \delta$.
These are as desired since any open ball intersecting the support of $\mu$ has positive measure.
\end{proof}

\begin{figure}[ht]\label{fig:me}
\centering

\tikzset{every picture/.style={line width=0.75pt}} 

\begin{tikzpicture}[x=0.75pt,y=0.75pt,yscale=-1,xscale=1]

\draw    (331,300) -- (328,-1) ;
\draw    (252,59) .. controls (275.4,22.92) and (356.79,24.88) .. (378.45,56.51) ;
\draw [shift={(380,59)}, rotate = 240.8] [fill={rgb, 255:red, 0; green, 0; blue, 0 }  ][line width=0.08]  [draw opacity=0] (8.93,-4.29) -- (0,0) -- (8.93,4.29) -- cycle    ;
\draw   (398,100) .. controls (398,86.19) and (409.19,75) .. (423,75) .. controls (436.81,75) and (448,86.19) .. (448,100) .. controls (448,113.81) and (436.81,125) .. (423,125) .. controls (409.19,125) and (398,113.81) .. (398,100) -- cycle ;
\draw   (490.5,188.75) .. controls (490.5,174.94) and (501.69,163.75) .. (515.5,163.75) .. controls (529.31,163.75) and (540.5,174.94) .. (540.5,188.75) .. controls (540.5,202.56) and (529.31,213.75) .. (515.5,213.75) .. controls (501.69,213.75) and (490.5,202.56) .. (490.5,188.75) -- cycle ;
\draw  [line width=0.75] [line join = round][line cap = round] (41.67,81.24) .. controls (44.79,81.24) and (48.56,78.08) .. (51.62,77.23) .. controls (59.34,75.1) and (71.52,77.16) .. (78.99,80.44) .. controls (88.04,84.41) and (91.85,86.45) .. (99.73,94.06) .. controls (112.03,105.95) and (116.83,135.25) .. (103.88,147.77) .. controls (97.68,153.75) and (83.63,156.18) .. (75.67,154.18) .. controls (59.34,150.09) and (72.01,132.21) .. (63.23,123.72) .. controls (58.83,119.47) and (50.61,114.38) .. (43.33,114.1) .. controls (36.14,113.83) and (26.85,118.22) .. (21.76,113.3) .. controls (20.89,112.46) and (14.46,106.2) .. (15.95,103.68) .. controls (20.26,96.4) and (35.51,82.04) .. (44.16,82.04) ;
\draw  [line width=0.75] [line join = round][line cap = round] (215.54,200.21) .. controls (206.68,200.21) and (198.37,199.69) .. (190.65,203.41) .. controls (177.28,209.87) and (167.06,227.91) .. (179.04,239.48) .. controls (187.18,247.35) and (201.77,247.02) .. (209.73,254.71) .. controls (214,258.84) and (214,266.89) .. (214.71,272.35) .. controls (216.02,282.49) and (224.19,291.18) .. (234.62,293.19) .. controls (237.97,293.84) and (241.8,295.6) .. (245.4,294.79) .. controls (253.57,292.97) and (260.04,286.27) .. (262.82,279.56) .. controls (264.92,274.49) and (264.03,267.46) .. (269.45,263.53) .. controls (280.17,255.76) and (296.26,264.62) .. (303.46,250.71) .. controls (310.31,237.47) and (302.63,217.71) .. (288.53,215.44) .. controls (283.18,214.58) and (276.8,215.71) .. (272.77,211.43) .. controls (264.34,202.47) and (267.54,201.11) .. (266.13,188.18) .. controls (265.64,183.62) and (253.69,177.25) .. (249.55,176.96) .. controls (242.51,176.48) and (236.29,177.11) .. (231.3,180.97) .. controls (221.69,188.39) and (225.34,200.21) .. (210.56,200.21) ;
\draw  [color={rgb, 255:red, 155; green, 155; blue, 155 }  ,draw opacity=1 ][line width=0.75] [line join = round][line cap = round] (99.61,115.96) .. controls (102,115.96) and (104.88,113.61) .. (107.22,112.99) .. controls (113.11,111.41) and (122.42,112.93) .. (128.12,115.36) .. controls (135.04,118.31) and (137.94,119.82) .. (143.96,125.46) .. controls (153.36,134.28) and (157.03,155.99) .. (147.13,165.28) .. controls (142.4,169.71) and (131.66,171.51) .. (125.59,170.03) .. controls (113.12,167) and (122.79,153.74) .. (116.09,147.45) .. controls (112.73,144.3) and (106.44,140.53) .. (100.88,140.32) .. controls (95.39,140.11) and (88.29,143.37) .. (84.41,139.72) .. controls (83.74,139.1) and (78.84,134.46) .. (79.97,132.59) .. controls (83.26,127.2) and (94.91,116.55) .. (101.52,116.55) ;
\draw  [color={rgb, 255:red, 155; green, 155; blue, 155 }  ,draw opacity=1 ][line width=0.75] [line join = round][line cap = round] (144.72,138.44) .. controls (146.59,138.44) and (148.85,136.28) .. (150.69,135.71) .. controls (155.32,134.25) and (162.63,135.65) .. (167.11,137.9) .. controls (172.54,140.61) and (174.82,142) .. (179.55,147.2) .. controls (186.93,155.32) and (189.81,175.33) .. (182.03,183.88) .. controls (178.32,187.97) and (169.89,189.62) .. (165.12,188.26) .. controls (155.32,185.47) and (162.92,173.25) .. (157.65,167.46) .. controls (155.01,164.55) and (150.08,161.08) .. (145.71,160.89) .. controls (141.4,160.7) and (135.83,163.7) .. (132.78,160.34) .. controls (132.25,159.77) and (128.4,155.49) .. (129.29,153.77) .. controls (131.88,148.8) and (141.02,138.99) .. (146.21,138.99) ;
\draw  [color={rgb, 255:red, 155; green, 155; blue, 155 }  ,draw opacity=1 ][line width=0.75] [line join = round][line cap = round] (156.51,171.68) .. controls (159.41,171.68) and (162.91,169.07) .. (165.75,168.38) .. controls (172.9,166.63) and (184.21,168.32) .. (191.14,171.02) .. controls (199.54,174.29) and (203.07,175.96) .. (210.38,182.23) .. controls (221.79,192.02) and (226.25,216.12) .. (214.23,226.42) .. controls (208.48,231.35) and (195.44,233.34) .. (188.06,231.7) .. controls (172.91,228.34) and (184.66,213.62) .. (176.52,206.64) .. controls (172.44,203.14) and (164.8,198.95) .. (158.05,198.72) .. controls (151.38,198.49) and (142.76,202.11) .. (138.04,198.06) .. controls (137.23,197.37) and (131.27,192.22) .. (132.65,190.15) .. controls (136.65,184.16) and (150.8,172.34) .. (158.82,172.34) ;
\draw  [color={rgb, 255:red, 155; green, 155; blue, 155 }  ,draw opacity=1 ] (423,125) .. controls (423,111.19) and (434.19,100) .. (448,100) .. controls (461.81,100) and (473,111.19) .. (473,125) .. controls (473,138.81) and (461.81,150) .. (448,150) .. controls (434.19,150) and (423,138.81) .. (423,125) -- cycle ;
\draw  [color={rgb, 255:red, 155; green, 155; blue, 155 }  ,draw opacity=1 ] (455.5,142.5) .. controls (455.5,132.84) and (463.34,125) .. (473,125) .. controls (482.66,125) and (490.5,132.84) .. (490.5,142.5) .. controls (490.5,152.16) and (482.66,160) .. (473,160) .. controls (463.34,160) and (455.5,152.16) .. (455.5,142.5) -- cycle ;
\draw  [color={rgb, 255:red, 155; green, 155; blue, 155 }  ,draw opacity=1 ] (473,163.75) .. controls (473,152.01) and (482.51,142.5) .. (494.25,142.5) .. controls (505.99,142.5) and (515.5,152.01) .. (515.5,163.75) .. controls (515.5,175.49) and (505.99,185) .. (494.25,185) .. controls (482.51,185) and (473,175.49) .. (473,163.75) -- cycle ;

\draw (141,13.4) node [anchor=north west][inner sep=0.75pt]    {$X$};
\draw (489,11.4) node [anchor=north west][inner sep=0.75pt]    {$Z$};
\draw (301,9.4) node [anchor=north west][inner sep=0.75pt]    {$\varphi $};
\draw (423,55.4) node [anchor=north west][inner sep=0.75pt]    {$B_{0}$};
\draw (550,176.4) node [anchor=north west][inner sep=0.75pt]    {$B_{1}$};
\draw (353,242.4) node [anchor=north west][inner sep=0.75pt]    {$d( B_{0} ,B_{1}) \ \leq \ 3\cdot \text{diam}( B_{0}) \ $};
\draw (76,58.4) node [anchor=north west][inner sep=0.75pt]    {$A\ =\ \varphi ^{-1} B_{0}$};
\draw (117,265.4) node [anchor=north west][inner sep=0.75pt]    {$B=\ \varphi ^{-1} B_{1}$};
\draw (133,92.4) node [anchor=north west][inner sep=0.75pt]  [color={rgb, 255:red, 155; green, 155; blue, 155 }  ,opacity=1 ]  {$\gamma _{1} A\ =\gamma _{1} \varphi ^{-1} B_{0} \ $};
\draw (187,124.4) node [anchor=north west][inner sep=0.75pt]  [color={rgb, 255:red, 155; green, 155; blue, 155 }  ,opacity=1 ]  {$\gamma _{2} A\ $};
\draw (203,151.4) node [anchor=north west][inner sep=0.75pt]  [color={rgb, 255:red, 155; green, 155; blue, 155 }  ,opacity=1 ]  {$\gamma _{3} A\ $};
\draw (466,83.4) node [anchor=north west][inner sep=0.75pt]  [color={rgb, 255:red, 155; green, 155; blue, 155 }  ,opacity=1 ]  {$ \  \ \gamma _{1} \cdot B_{0} \ $};
\draw (492,110.4) node [anchor=north west][inner sep=0.75pt]  [color={rgb, 255:red, 155; green, 155; blue, 155 }  ,opacity=1 ]  {$\gamma _{2} \cdot B_{0} \ $};
\draw (517,138.4) node [anchor=north west][inner sep=0.75pt]  [color={rgb, 255:red, 155; green, 155; blue, 155 }  ,opacity=1 ]  {$\gamma _{3} \cdot B_{0} \ $};

\end{tikzpicture}

\caption{$k$-chaining implies metric ergodicity with $k=3$}
\end{figure}
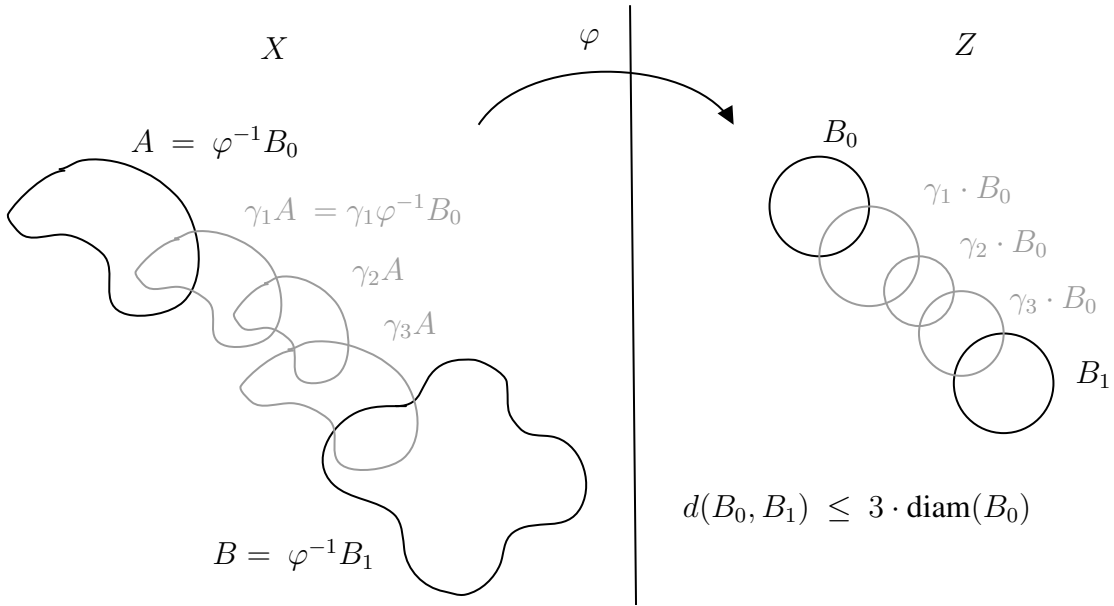

\begin{theorem}\label{essential_bounded_chaining_implies_ME}
Let $\Gamma$ be a countable group, $(X,\mu)$ be a standard probability space, and $\Gamma \actson^\alpha (X,\mu)$ be an mcp action.
If the action $\alpha$ is essentially boundedly chaining then it is metrically ergodic.
\end{theorem}

\begin{proof}
We assume that the action $\alpha$ is essentially $k$-chaining for some $k \in \N$, so let $W$ be as in \cref{defn:ess-bounded_chaining}. 
Suppose toward a contradiction that there is a non-essentially constant equivariant measurable map $\phi : X \to Z$ to a separable metric space $(Z,d)$ on which $\Gamma$ acts by isometries.
Notice that since $\bigcup \gamma W =_\mu X$, $\phi\rest{W}$ is still non-essentially constant.

Because the pushforward measure $\phi_* (\mu\rest{W})$ on $Z$ is not a Dirac measure,
\cref{ball-separation} yields disjoint balls $B_0, B_1 \subseteq Z$ of positive $\phi_* (\mu\rest{W})$-measure such that $d(B_0,B_1) > k \cdot \diam(B_0)$.
Then $A \defeq \phi^{-1}(B_0)$ and $B \defeq \phi^{-1}(B_1)$ are measurable subsets of $X$ of positive $\mu\rest{W}$-measure.

Let $\gamma_1, \gamma_2, \dots \gamma_k$ witness that $A$ $k$-chains to $B$, \ie $\gamma_1 A \cap A >_{\mu} 0$, $\gamma_i A \cap \gamma_{i-1} A >_{\mu} 0$ for each $i \ge 2$, and $\gamma_k A \cap B >_{\mu} 0$.

Now for all $1 \le i \le k$ and a.e. $x \in X$:
\begin{align*}
x \in \phi^{-1}(\gamma_i B_0) 
&\iff  \gamma_i^{-1} \phi(x) \in B_0
\\
\eqcomment{equivariance}
&\iff 
\phi(\gamma_i^{-1} x) \in B_0
\\
&\iff 
x \in \gamma_i \phi^{-1}(B_0) = \gamma_i A,
\end{align*}

so $\phi^{-1}(\gamma_i B_0) =_{\mu} \gamma_i A $.
Thus if $\gamma_i A \cap D >_{\mu} 0$ for some measurable set $D \subseteq X$, then also $\phi^{-1}(\gamma_i B_0) \cap D >_{\mu} 0$, so $\gamma_i B_0 \cap \phi(D) \supseteq \phi(\phi^{-1}(\gamma_i B_0) \cap D) >_{\phi_* \mu} 0$.
Thus, the $\phi$-images $B_0$ and $B_1$ of $A$ and $B$ are also chained, \ie $\gamma_1 B_0 \cap B_0 \neq \emptyset$, $\gamma_i B_0 \cap \gamma_{i-1} B_0 \neq \emptyset$ for each $i \ge 2$, and $\gamma_k B_0 \cap B_1 \neq \emptyset$.

But since the action of $\Gamma$ on $(Z,d)$ is by isometries, we have $\diam(\gamma_i B_0) = \diam(B_0)$ so $d(B_0, B_1) \le k \cdot \diam(B_0)$, a contradiction.
\end{proof}

From \cref{essential_bounded_chaining_implies_ME} and \cite[Theorem 1.1-1.2]{Glasner-Weiss:weak_mixing} (and the fact that for pmp actions, weak mixing implies double ergodicity) we immediately obtain the following characterization.

\begin{cor}\label{pmpwmchaining}
Let $(X,\mu)$ be a standard probability space, and let $\Gamma\actson^\alpha (X,\mu)$ be a pmp action of a countable group $\Gamma$.
Then the following are equivalent:
\begin{enumerate}
\item $\alpha$ is weakly mixing.
\item $\alpha$ is 1-chaining.
\item $\alpha$ is boundedly chaining.
\end{enumerate}
\end{cor}

\section{Bounded chaining for boundary actions of free groups}\label{sec:boundary_actions}

This section is devoted to the natural action of the free group on its boundary, equipped with a Markov measure.
We refer to \cref{subsec:Markov_chains,subsec:pos_stationary_distrib,subsec:boundary} for definitions, notation, and statements regarding Markov measures and boundary actions.
We collect our running notation and assumptions into the following hypothesis, which we assume throughout this section, in order to declutter the statements.

\begin{hypothesis}\label{hyp}
Let
\begin{itemize}[leftmargin=*]
\item $1 \le r < \infty$;
\item $\F_r$ denote the free group of rank $r$ on set $\set{a_1,\dots,a_r}$ of free generators;
\item $\Symb \defeq \set{a_1^{\pm 1}, \dots, a_r^{\pm 1}}$;
\item $\mu$ be a Markov measure on $\partial \F_r$, whose stochastic matrix $P$ admits a stationary distribution; 
in particular, for all $a \in \Symb, P(a,a^{-1}) = 0$;
\item $L$ be the set of $\mu$-legal words, \ie 
\[
L = \set{x \in \partial \F_r : P(x_i, x_{i+1}) > 0 \; \text{ for all } i\in\N};
\]
\item $\F_r \actson^\beta (\partial \F_r,\mu)$ be the boundary action;
\item $\tilde{\mu}$ be a nonsingularization of $\mu$ given by $\tilde{\mu} = \sum_{n \ge 1} 2^{-n} ({\gamma_n})_* \mu$, where $(\gamma_n)_{n \ge 1}$ is an arbitrary enumeration of $\F_r$.
\end{itemize}
\end{hypothesis}

\noindent Furthermore, for each $w \in \Words$, we abuse notation and denote $[w] \defeq \set{x \in \partial \F_r : x \text{ extends } w}$.

\subsection{Charactarizations of weakly mixing boundary actions}

In this section, we prove \cref{intro:wm_boundary}, restated below as \cref{wm_boundary}.

\begin{lemma}\label{L_space-group-generating}
Assume \cref{hyp}.
Then $L$ is a group generating complete section with respect to the nonsingularization $\F_r \actson (\partial \F_r, \~\mu)$.
\end{lemma}
\begin{proof}
That $L$ is a complete section follows directly from the definition of $\~\mu$ in \cref{hyp} and the fact that $L$ is $\mu$-conull.

The set $L$ is also group generating because each $a \in S$ is a return time for $L$.
Indeed, let $b \in \Symb$ with $P(a^{-1},b) > 0$. 
Then $\tilde{\mu}(a L \cap L) \ge \mu(a L \cap L) \ge \mu([b]) > 0$.
\end{proof}

\begin{lemma}\label{Ess-bounded_chaining_for_nonsingularization}
Assume \cref{hyp}.
Then the boundary action $\F_r \actson^\beta (\partial \F_r,\mu)$ is $k$-chaining if and only if its nonsingularization $\F_r \actson^\beta (\partial \F_r,\~\mu)$ is $k$-chaining on $L$.
In particular, if $\F_r \actson^\beta (\partial \F_r,\mu)$ is $k$-chaining, then $\F_r \actson^\beta (\partial \F_r,\~\mu)$ is essentially $k$-chaining.
\end{lemma}

\begin{proof}
The ``in particular'' part follows from \cref{L_space-group-generating}.
It is straightforward to verify that for $A \subseteq L$ and $\gamma \in \F_r$ with $\gamma A \subseteq L$, we have $\mu(\gamma A) > 0$ exactly when $\mu(A) > 0$, so $E_\beta\rest{L}$ is $\mu$-mcp.
It then follows by the definition of $\~\mu$ in \cref{hyp} that $\mu\rest{L} \sim \~\mu \rest{L}$, so existence of $k$-chains on subsets of $L$ holds or fails simultaneously with respect to the two measures $\mu$ and $\~\mu$.
The conclusion now follows because $L$ is $\mu$-conull.
\end{proof}

\begin{lemma}\label{n_dance_n_chaining}
Assume \cref{hyp}.
If $P$ is strictly irreducible, then $\F_r \actson (\partial \F_r, \mu)$ is $k$-chaining, where $k$ is the diameter of the (undirected) graph $G_{P^t P}$.
\end{lemma}

\begin{proof}
Let $A, B \subseteq \partial \F_r$ have positive measure. 
Then there are $a, b \in \Symb$ and $\frac{1}{2} > \e > 0$ such that $\mu_a(A) > \e$ and $\mu_b(B) > \e$. 
By assumption, there is a finite sequence $a_1, b_1, a_2, b_2, \dots, b_{k-1}, a_k \in \Symb$ such that $a_1 = a$, $a_k = b$ and for each $i < k$,
\[
P(b_i, a_i) > 0 \text{ and } P(b_i, a_{i+1}) > 0.
\tag{$\ast$}
\label{dance_moves}
\]
Because $\Symb$ is finite, there is a $\delta > 0$ such that for all Borel sets $C \subseteq \partial \F_r$, and all $P$-legal words $ws$ with $w \in \Words$ and $s \in \Symb$, we have
\[
\mu_w(C) > 1 - \delta \implies \mu_{ws}(C) > 1 - \e.
\]

For each $i<k$, Lebesgue density (\cref{Lebesgue_density}) gives a $P$-legal word $w_i b_i$ with $w_i \in \Words$ such that $\mu_{w_i b_i}(A) > 1 - \delta$, so by \labelcref{dance_moves}, we get that both $\mu_{w_i b_i a_i}(A)$ and $\mu_{w_i b_ia_{i+1}}(A)$ are greater than $1 - \e$. 
Now the Markov property (\cref{MP}) yields
\[
\mu_{a_i}((w_i b_i)^{-1}A) > 1 - \e \; \text{ and } \; \mu_{a_{i+1}}((w_i b_i)^{-1}A) > 1-\e.
\]
Hence $\mu((w_i b_i)^{-1} A \cap (w_{i+1} b_{i+1})^{-1} A) > 0$ for each $1 \le i < k-1$, 
as well as $\mu(A \cap (w_1 b_1)^{-1} A) > 0$ and $\mu((w_{k-1} b_{k-1})^{-1} A \cap B) > 0$ since $a_1 = a$ and $a_k = b$. 
Thus $\gamma_i \defeq w_i b_i$ is as desired.
\end{proof}

\begin{defn}\label{symmetry}
Let $\Symb$ and $P$ be as in \cref{hyp}.
We say that $P$ is \dfn{weakly inverse-symmetric} if for all $a,b \in \Symb$ with $b \ne a^{\pm 1}$, we have 
\[
P(a,b) = 0 \iff P(a^i,b^j) = 0 \text{ for all } i,j \in \set{\pm 1}.
\]
\end{defn}

\begin{theorem}[Weak mixing criteria for boundary actions]\label{wm_boundary}

Assume \cref{hyp}. 
Suppose in addition that $P$ is weakly inverse-symmetric.
Then the following are equivalent:

\smallskip

\begin{enumerate}[(1),leftmargin=*]
\item\label{item:wm} The boundary action $\F_r \curvearrowright^\bdryaction (\partial \F_r,\mu)$ is weakly mixing.

\item\label{item:me} The nonsingularization $\F_r \curvearrowright^\bdryaction (\partial \F_r,\~\mu)$ is metrically ergodic.

\item\label{item:bounded_chaning} The boundary action $\F_r \curvearrowright^\bdryaction (\partial \F_r,\mu)$ is boundedly chaining.

\item\label{item:ess(S-1)-chaning} The nonsingularization $\F_r \curvearrowright^\bdryaction (\partial \F_r,\tilde{\mu})$ is essentially $(|S|-1)$-chaining.

\item\label{item:(S-1)-chaning} The boundary action $\F_r \curvearrowright^\bdryaction (\partial \F_r,\mu)$ is $(|S|-1)$-chaining.

\item\label{item:strictly-irreducible} $P$ is strictly irreducible.
\end{enumerate}
\end{theorem}

\begin{proof}
The implication $\labelcref{item:strictly-irreducible} \imp
\labelcref{item:(S-1)-chaning}$ is immediate from \cref{n_dance_n_chaining}, strict irreducibility, and $|S|$.
$\labelcref{item:(S-1)-chaning}
\imp \labelcref{item:ess(S-1)-chaning}$ is due to \cref{Ess-bounded_chaining_for_nonsingularization}. 
To see $\labelcref{item:ess(S-1)-chaning} \imp \labelcref{item:wm}$, note that by $\labelcref{item:ess(S-1)-chaning}$, the boundary action is weakly mixing with respect to $\tilde{\mu}$ (this is by \cref{Ess-bounded_chaining_for_nonsingularization}, \cref{essential_bounded_chaining_implies_ME}, and \cite[Theorem 1.1-1.2]{Glasner-Weiss:weak_mixing}).
Then by \cref{WM_nonsingularization}, the boundary action is weakly mixing with respect to $\mu$.
$\labelcref{item:(S-1)-chaning} \imp \labelcref{item:bounded_chaning}$ is trivial, 
\labelcref{item:bounded_chaning} $\imp$ \labelcref{item:me} is by \cref{essential_bounded_chaining_implies_ME} and \cref{Ess-bounded_chaining_for_nonsingularization}, 
and $\labelcref{item:me} \imp \labelcref{item:wm}$ is by \cite[Theorem 1.1-1.2]{Glasner-Weiss:weak_mixing} and \cref{WM_nonsingularization}.

\labelcref{item:wm} $\imp$ \labelcref{item:strictly-irreducible} follows from \cref{wm_implies_strict_irreducible} below, since the action $\F_r \actson^\alpha (\Z_2^V, \nu)$ is ergodic and pmp.
Note that this is the only use of the weak-inverse symmetry hypothesis on $P$.
\end{proof}

\begin{question}
    Which conditions in \cref{wm_boundary} are equivalent without the weak inverse-symmetry assumption?
\end{question}

\subsubsection{Weak mixing implies strict irreducibility}\label{subsubsec:wm_imp_strict-irred}

Assume \cref{hyp} and in addition that $P$ is weakly inverse-symmetric.
We now define a specific ergodic pmp action $\F_r \actson^\alpha (Z,\nu)$ and show that the ergodicity of the diagonal action $\F_r \actson (\partial \F_r \times Z, \mu \times \nu)$ implies that the transition matrix $P$ is strictly irreducible.

Given a directed graph $G = (V, E)$, we denote by $\Z_2^V$ the group of formal sums of elements of $V$ with addition mod 2.
We call an element $g \in \Z_2^V$ a \dfn{$G$-cycle} if there is a directed path $v_0, v_1, \dots, v_n$ in $G$ with $v_0 = v_n$ and $g = v_1 + v_2 + \dots + v_n$.
We denote by $\Cyc$ the subgroup of $\Z_2^V$ generated by the $G$-cycles.

We now define a specific graph $G$ as follows.
Let $\Symb^+ \defeq \set{a_1,a_2,\dots,a_r}$, \ie the set of \emph{positive} free generators of $\F_r$.
Let $G = (V , E)$ be the quotient graph of $G_P$ where we identify a generator with its inverse, namely:
the vertex set is $V = \set{[a]_R: a \in \Symb^+}$, where $[a]_R = \set{a,a^{-1}}$, and $([a]_R,[b]_R) \in E$ exactly when there is an edge in $G_P$ of the form $(a^i, b^j)$ for some $i,j \in \set{\pm 1}$.

Let $\iota : \F_r \to \Z_2^V$ be the unique group homomorphism induced by mapping each $a \in \Symb$ to $[a]_R \in V \subseteq \Z_2^V$, that is, $\iota(w) \defeq \sum_{i=1}^n [a_i]_R$ for a word $w = a_1 \dots a_n \in \F_r$.

We equip $Z \defeq \Z_2^V$ with the uniform probability measure $\nu$ and define the action $\F_r \actson^\alpha \Z_2^V$ by $\gamma \cdot^\alpha v \defeq \iota(\gamma) + v$, so $\alpha$ is ergodic (transitive) and pmp.
For the rest of the section, assume the diagonal action $\F_r \actson^{\beta \times \alpha} (\partial \F_r \times \Z_2^V, \mu \times \nu)$ is ergodic, and hence $P$ is irreducible.

Recall that a directed graph is \dfn{strongly connected} if for any two vertices $v$ and $w$, there is a directed path from $v$ to $w$.
Observe that since $P$ is irreducible, $G$ is strongly connected.

\begin{lemma}\label{lem: WM implies Cyc = everything}
$\Cyc=\Z_2^V$.
\end{lemma}

\begin{proof}
Consider the diagonal action $\F_r \actson^{\bdryaction\times \alpha} \partial \F_r \times \Z_2^V$. 
Fix $a \in \Symb$ and, so $a \in \Rec$ since $P$ is irreducible. 
Let
\[
A \defeq \set{(x, w): x \in [a] \text{ and } w \in \Cyc}.
\]

Consider the $\bdryaction\times\alpha$-saturation of $A$, which is necessarily invariant under the action. 
Letting $\dom(\gamma) \defeq \set{x\in \Symb^\N \ : \ \mu[\gamma x_0]>0}$ for each $\gamma \in \F_r$, we make the following assertion.

\begin{claim*}
$[A]_{\bdryaction \times \alpha} =_{\mu \times \nu} \set{(\delta \cdot y, \delta \cdot v) : \delta \in \F_r,\; y \in [a] \cap \dom(\delta), v \in \Cyc}.$
\end{claim*}

\begin{pf}
First, note that the points in the $\bdryaction\times\alpha$-saturation of $A$ are exactly those of the form $(\gamma \cdot (x \rest{n})^{-1} \cdot x, \gamma \cdot (x \rest{n})^{-1} \cdot w)$, where $\gamma \in \F_r$, $x \in [a]$, $n \in \N$, and $w \in \Cyc$.
Discarding a null set, we may also assume that $\sigma^n(x)\in \dom(\gamma)$, $\gamma \cdot \sigma^n(x)$ is a legal word, and that $x_m = a$ for infinitely many $m \in \N$ (since $a \in \Rec$).
Note that for every such $m$, the vertices in $\sigma^m(x)$ form a $G$-cycle, so $v \defeq (x \rest{m})^{-1} \cdot w = \iota(x\rest{m}) + w \in \Cyc$.
Hence, by taking such an $m$ larger than $n$ above and setting $y \defeq \sigma^m(x)$ and $\delta \defeq \gamma \cdot (x \rest{[n,m)})$, we may assume that our points are of the form $(\delta \cdot y, \delta \cdot v)$, where $y \in [a] \cap \dom(\delta)$ and $v \in \Cyc$.
\end{pf}

Thus, up to measure zero, every pair of the form $(x,w)\in [A]_{\bdryaction\times \alpha}$ with $x \in [a]$ is such that $w\in \Cyc$. 
Therefore, since $\bdryaction\times \alpha$ is ergodic and $[A]_{\bdryaction\times \alpha}$ is invariant, we must have $\Cyc = \Z_2^V$.
\end{proof}

For $[a]_R,[b]_R \in V$, put $[a]_R \sim [b]_R$ if $a \reach b$.

\begin{lemma}
The relation $\sim$ on $V$ is a well-defined equivalence relation.
\end{lemma}
\begin{proof}
By the symmetry and transitivity of the relation $\reach$ on $\Symb$, it is enough to show that $\sim$ is well-defined and reflexive.
To show well-definedness, suppose that $a \reach b$ for some $a,b \in \Symb$. Note that irreducibility implies that $P(c,a) > 0$ for some $c \in \Symb$ not equal to $a$ or $a^{-1}$. 
Then our symmetry condition \labelcref{symmetry} yields that $P(c, a^{-1}) > 0$ so $a^{-1} \reach a$ and hence $a^{-1} \reach b$.
Similarly, we get $a \reach b^{-1}$ by the symmetry of $\reach$.

The reflexivity of $\sim$ also follows from the irreducibility of $P$ since for each $a \in \Symb$ we have $P(c,a) > 0$ for some $c \in \Symb$.
\end{proof}

\begin{obs}\label{PtP-relation}
Let $G = (V,E)$ be as defined above.
\begin{enumerate}[(a),leftmargin=*]

\item\label{out-neighbors_PtP-reachable} All out-neighbors of a vertex are in the same $\sim$ equivalence class.

\item \label{irreducible_iff_outneighbors} 
$P^tP$ is irreducible if and only if every vertex in $V$ is $\sim$ to all of its out neighbors.
\end{enumerate}
\end{obs}

\begin{lemma}\label{main_lemma}
If a vertex $[a]_R \in V$ is in $\Cyc$, then $[a]_R \sim$ to all of its out-neighbours.
\end{lemma}

\begin{proof}
Let $[a]_R = \sum_{j<k} C_j$, where each $C_j \in \Cyc$ and let $\tilde{H}$ be the directed multigraph (with possibly repeated directed edges) on the vertex set $V$ formed by the union the $C_j$ viewed as $G$-cycles in the Markov graph $G$. 
In $\tilde{H}$, each vertex has out-degree $=$ in-degree because this is true of each $C_j$. 
Moreover, because in $\sum_{j<k} C_j$ all labels other than $[a]_R$ cancel and $[a]_R$ remains, every vertex other than $[a]_R$ has even out-degree (hence also even in-degree) in $\tilde{H}$ while $[a]_R$ has odd out-degree (hence also odd in-degree) in $\tilde{H}$. 

Let $H$ be the graph obtained from $\tilde{H}$ by changing the multiplicity $m_e$ of each directed edge $e$ to $m_e \mod 2$, so $H$ is a subgraph of $G$ in which the parity of out and in degrees of each vertex are the same as in $\tilde{H}$.
Thus, in the graph $H$, each vertex $v \in \Symb$ has equal in and out degrees, and its out-degree is odd exactly when $v = [a]_R$.

For any directed graph $K$, in addition to the standard notations of the vertex set $V(K)$ and edge set $E(K)$, we also denote by $O(K)$ and $I(K)$ the sets of vertices of $K$ with positive out and in degrees, respectively.
Call a subgraph $K$ of $H$ \emph{$[a]_R$-zigzagged} if it is connected as an undirected graph, $[a]_R \in V(K)$, and each vertex $v \in V(K)$ satisfies the following:

\begin{enumerate}[(i),leftmargin=*]
\item\label{zigzag-out} if $v \in O(K)$, then there are $n \ge 0$ and a trail in $K$ of the form
\[
[a]_R = a_0 \rightarrow b_1 \leftarrow a_1 \rightarrow \dots \leftarrow a_{n-1} \rightarrow b_n \leftarrow a_n = v,
\]

\item\label{zigzag-in} and if $v \in I(K)$, then there are $n \ge 0$ and a trail in $K$ of the form 
\[
[a]_R = a_0 \rightarrow b_1 \leftarrow a_1 \rightarrow \dots \leftarrow a_n \rightarrow b_{n+1} = v,
\]
\end{enumerate}
where the arrows signify directed edges of $K$.
For such a graph $K$, observe that:

\begin{enumerate}[(a),leftmargin=*]
\item \label{a_in_O(K)} $[a]_R \in O(K)$ unless $E(K) = \emptyset$;

\item \label{B_interreachable} $I(K)$ is contained in one $\sim$ equivalence class.
\end{enumerate}

Because $[a]_R$-zigzagged subgraphs of $H$ are closed under unions, the union of all $[a]_R$-zigzagged subgraphs of $H$ is the unique maximal $[a]_R$-zigzagged subgraph, and we denote it by $H_{[a]_R}$.
By \labelcref{B_interreachable}, to show that $[a]_R \sim$ to all its out-neighbors, it is enough to prove that $[a]_R \in I(H_{[a]_R})$.

It follows from maximality that in $H_{[a]_R}$, the out-degrees of vertices in $O(H_{[a]_R})$ and the in-degrees of vertices in $I(H_{[a]_R})$ coincide with those in $H$.
Moreover, maximality also implies that the out-degree of $[a]_R$ in $H_{[a]_R}$ coincides with that in $H$, and since it is odd in $H$, we have $[a]_R \in O(H_{[a]_R})$.
This and the fact that $[a]_R$ is the only vertex in $H$ with odd out-degree yields that the sum of out-degrees of all vertices in $H_{[a]_R}$ is odd, and thus same is true for in-degrees.
But $[a]_R$ is the only vertex with odd in-degree in $H$, so it must be that $[a]_R \in I(H_{[a]_R})$.
\end{proof}

Thus, \cref{PtP-relation}\labelcref{irreducible_iff_outneighbors}, and \cref{main_lemma,lem: WM implies Cyc = everything} together yield the desired result:

\begin{prop}\label{wm_implies_strict_irreducible}
Assume \cref{hyp}. 
Suppose in addition that $P$ is weakly inverse-symmetric.
Let $\F_r \actson^\alpha (\Z_2^V, \nu)$ be the ergodic pmp action defined above.
If the diagonal action $\F_r \actson^{\beta \times \alpha}(\partial \F_r \times \Z_2^V, \mu \times \nu)$ is ergodic then $P$ is strictly irreducible.
\end{prop}

\subsection{Counterexamples involving boundary actions}\label{sec:counterexamples}

Here we will show that in the mcp setting, essential $1$-chaining does not imply bounded chaining (\cref{ess_1C_not_BC}), and essential $k+1$-chaining does not imply essential $k$-chaining (\cref{strict_essBC_hierarchy}), \ie the essential bounded chaining hierarchy is strict:
\begin{align*}
\text{1-ess-C}
\mathop{\substack{
\imp
\\
\red{\nLeftarrow}
}}
\text{2-ess-C}
\mathop{\substack{
\imp
\\
\red{\nLeftarrow}
}}
\text{3-ess-C}
\mathop{\substack{
\imp
\\
\red{\nLeftarrow}
}}
\dots
\mathop{\substack{
\imp
\\
\red{\nLeftarrow}
}}
\text{ess-BC}.
\end{align*}

\noindent We also give similar examples of singular actions that are $k+1$-chaining but not $k$-chaining.
It remains open whether such examples exist in the mcp setting, see \cref{q:bbd-C_implies_1-C}.
We note that all of our counterexamples in this section are for free group actions. 
Thus, it is still possible that these implications hold for amenable groups, see \cref{q:amenable-WM_implies_bdd-C}.

\begin{prop}\label{ess_1C_not_BC}
There is an mcp action of $\F_2$ that is essentially 1-chaining but not boundedly chaining.
\end{prop}

\begin{proof}
Let $a$ and $b$ be the free generators of $\F_2$, and let $\mu$ be a Markov measure on $\partial \F_2$ whose transition matrix is symmetric and satisfies $P(a^i, b^j) > 0$, $P(b^i,b^i) > 0$, and $P(a^i,a^i) = 0$ for all $i,j \in \set{\pm 1}$, \ie the associated graph $G_P$ is as follows:

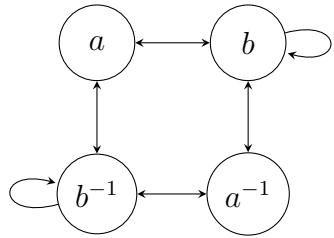
\begin{center}
\begin{tikzpicture}[>=stealth, every node/.style={circle, draw, minimum size=1cm}]
\node (b) at (0,2) {$a$};
\node (a) at (2,2) {$b$};
\node (ainv) at (0,0) {$b^{-1}$};
\node (binv) at (2,0) {$a^{-1}$};

\draw[<->] (a) -- (b);
\draw[<->] (a) -- (binv);
\draw[<->] (ainv) -- (b);
\draw[<->] (ainv) -- (binv);

\draw[->] (a) edge[loop right] (a);
\draw[->] (ainv) edge[loop left] (ainv);
\end{tikzpicture}
\end{center}

Since $P^tP(s,t)>0$ for all $s,t \in \Symb$, by \cref{n_dance_n_chaining}, $\bdryaction$ is $1$-chaining with respect to $\mu$.
Then by \cref{Ess-bounded_chaining_for_nonsingularization}, $\bdryaction$ is essentially $1$-chaining with respect to $\tilde{\mu}$.

It remains to show that $\bdryaction$ is not boundedly chaining with respect to $\tilde{\mu}$.
To see this, let $L$ denote the set of all $P$-legal infinite words in $\Symb^\N$ and let $B_n \defeq [a^{2n+2}]$ for all $n \in \N$.
We claim that for all $n \in \N$, $L$ does not $n$-chain to $B_n$.
To show this, we will show that even if we only require that the intersections of the chain are nonempty (rather than having positive measure), there is no $n$-chain from $L$ to $B_n$.
For $n=0$, simply note that $L \cap B_0 = \emptyset$ since $P(a,a) = 0$.

For $n=1$, we show that for all $\gamma \in \F_2$, $\gamma L \cap L = \emptyset$ or $\gamma L \cap [aaaa] = \emptyset$.
If $\gamma L \cap L \neq \emptyset$, then $\gamma$ cannot begin with $aaa$, since $P(a,a) = P(a^{-1},a^{-1}) = 0$.
Therefore, $\gamma L \cap [aaaa] = \emptyset$ so $\gamma L \cap B_1 = \gamma L \cap [aaaa] = \emptyset$.

We will now show inductively that $L$ does not $k$-chain to $[a^{2k+2}]$.
By the above argument, $L$ does not $1$-chain to $[a^4]$.
Suppose now that we have a $(k+1)$-chain from $L$ witnessed by $\gamma_1, \dots , \gamma_k, \gamma_{k+1}$.
By the inductive hypothesis, $\gamma_k L \cap [a^{2k+2}] = \emptyset$, so each $x \in \gamma_k L$ does not begin with $a^{2k+2}$.
Since $\gamma_{k} L \cap \gamma_{k+1} L \neq \emptyset$ and $P(a^{-1}, a^{-1}) = 0$ (so at most one $a$ can cancel), the word $\gamma_{k+1}$ cannot begin with $a^{2k+3}$, so each $x \in \gamma_{k+1} L$ does not begin with $a^{2k+4}$.
Therefore, $\gamma_{k+1}L \cap [a^{2k+4}] = \emptyset$, as desired.
\end{proof}

\begin{remark}\label{remark:chaining_is_not_symmetric}
The above example also shows that the relation ``$A$ $k$-chains to $B$'' is not symmetric: while the set $L$ does not $1$-chain to $[a^4]$, the set $[a^4]$ \textit{does} $1$-chain to $L$ via $\gamma = a^{-4}$.
Indeed, $\gamma[a^4] \supseteq \gamma[a^5] = [a]$, so $\~\mu(\gamma[a^4] \cap L) > 0$.
We also have that $\gamma[a^4] \supseteq \gamma[a^8] = [a^4]$, so $\~\mu(\gamma[a^4] \cap [a^4]) > 0$.
\end{remark}

\begin{prop}\label{strict_essBC_hierarchy}
Let $k \in \N$.
There is a singular (non-mcp) Borel action $\F_{2k+2} \actson (X,\mu)$ which is $(k+1)$-chaining but not $k$-chaining.
Furthermore, the nonsingularization $\F_{2k+2} \actson (X,\~\mu)$ is essentially $(k+1)$-chaining but not essentially $k$-chaining, where $\tilde{\mu} = \sum_{n \ge 1} 2^{-n} ({\gamma_n})_* \mu$ and $(\gamma_n)_{n \ge 1}$ is an arbitrary enumeration of $\F_{2k+2}$.
\end{prop}

\begin{proof}
The following diagram shows a diagram of $G_P$ where the boundary action $\F_{2k+2} \actson^\bdryaction (\partial\F_{2k+2}, \mu)$ is $(k+1)$-chaining but not $k$-chaining.
Below, for $1\le j \le 2n+2$, the node labeled $j$ contains the vertices $a_j^{\pm 1}$, and the arrows represent all possible arrows between the vertices within the corresponding nodes, aside from transitioning from a state to its inverse.

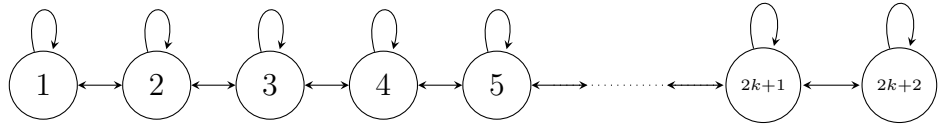
\begin{center}
\begin{tikzpicture}[
>=stealth,
every node/.style={circle, draw, minimum size=9mm},
scale=1.5
]

\node (1)  at (0,0)  {$1$};
\node (2)  at (1,0)  {$2$};
\node (3)  at (2,0)  {$3$};
\node (4)  at (3,0)  {$4$};
\node (5)  at (4,0)  {$5$};
\draw [dotted] (4.5,0) -- (6,0);
\node (m) at (6.35,0) {{$\scriptscriptstyle 2k+1$}};
\node (n) at (7.55,0) {{$\scriptscriptstyle 2k+2$}};

\draw[<->] (1) to (2);
\draw[<->] (2) to (3);
\draw[<->] (3) to (4);
\draw[<->] (4) to (5);
\draw[<->] (5) to (4.8,0);
\draw[<->] (5.5,0) to (6,0);
\draw[<->] (6.7,0) to (7.2,0);
\draw[->] (1) edge[loop above] (1);
\draw[->] (2) edge[loop above] (2);
\draw[->] (3) edge[loop above] (3);
\draw[->] (4) edge[loop above] (4);
\draw[->] (5) edge[loop above] (5);
\draw[->] (m) edge[loop above] (m);
\draw[->] (n) edge[loop above] (n);
\end{tikzpicture}
\end{center}

One can check that for any two states $a$ and $b$, $a \dancesto b$ can be witnessed in at most $(k+1)$ steps, so by \cref{Ess-bounded_chaining_for_nonsingularization} and \cref{n_dance_n_chaining}, $\bdryaction$ is $(k+1)$-chaining with respect to $\mu$ and thus essentially $(k+1)$-chaining with respect to $\tilde{\mu}$.

We now show that $\bdryaction$ is not $k$-chaining with respect to $\mu$.

\begin{claim}\label{claim:not_k_chain}
    $A = [a_1] \cap L$ does not $k$-chain to $ B = [ a_{2k+2}] \cap L$ with respect to $\mu$.
\end{claim}

\begin{pf}
    We begin by noting that if $A \cap \gamma A >_\mu 0$, then $\gamma A \subseteq \bigcup_{j=1}^3 ([a_j] \cup [a_j^{-1}])$: indeed, since $A \cap \gamma A >_\mu 0$, 
$\gamma_0 \in \set{a_1^{\pm 1}, a_2^{\pm 1}}$.
Hence, if $\gamma A \cap [b] >_\mu 0$, then $b \in \set{a_1^{\pm 1}, a_2^{\pm 1}, a_3^{\pm 1}}$.
Iterating this, we see that for all $j \le k$ and $i > 2j+1$, $[a_1] \cap L$ does not $j$-chain to $[a_{i}] \cap L$.
In particular, $A$ does not $k$-chain to $B$.
\end{pf}

Thus, the boundary action is not $k$-chaining with respect to $\mu$.
It remains to show that the action is not essentially $k$-chaining with respect to $\~\mu$.
Let $W >_{\~\mu} 0$.
Then there is $\gamma \in \Gamma$ for which $\gamma W >_\mu 0$.
Notice that if $A$ $k$-chains to $B$ with respect to $\~\mu$, then the same group elements also witness that $\gamma^{-1} A$ $k$-chains to $\gamma^{-1} B$ with respect to $\~\mu$ (since $\~\mu$ is mcp).

\begin{claim}\label{claim:long_words}
For any legal words $w$ and $v$ of length $n$ such that $n=0$ or the transitions $P(w_{n-1},a_1)$ and $P(v_{n-1},a_{2k+2})$ are positive, $A = [w\conc a_1] \cap L$ does not $k$-chain to $ B = [v \conc a_{2k+2}] \cap L$ with respect to $\mu$.
\end{claim}

\begin{pf}
Set $A = [w\conc a_1] \cap L$ and $ B = [v \conc a_{2k+2}] \cap L$.
By the same argument as in \cref{claim:not_k_chain}, if $A \cap \gamma A>_\mu 0$, then $\gamma_n \in \set{a_1^{\pm 1}, a_2^{\pm 1}}$.
Hence, if $\gamma A \cap [\delta \conc b] >_\mu 0$ for $|\delta| = n$ and $b \in \Symb$, then $b \in \set{a_1^{\pm 1}, a_2^{\pm 1}, a_3^{\pm 1}}$. 
Again, iterating this shows that $A = [w\conc a_1] \cap L$ does not $k$-chain to $ B = [v \conc a_{2k+2}] \cap L$.
\end{pf}

Let $w$ and $v$ be words of length $n$ such that $\mu(\gamma W \cap [w \conc a_1]) > 0$ and  $\mu(\gamma W \cap [v \conc a_{2k+2}]) > 0$.
Then we also have $\mu(\gamma W \cap [w \conc a_1] \cap L) > 0$ and $\mu(\gamma W \cap [v \conc a_{2k+2}] \cap L) > 0$.
By \cref{claim:long_words}, $A = \gamma W \cap [w a_1] \cap L$ does not $k$-chain to $B = \gamma W \cap [w a_{2k+2}] \cap L$ with respect to $\mu$, so by \cref{Ess-bounded_chaining_for_nonsingularization}, $A = \gamma W \cap [w a_1] \cap L$ does not $k$-chain to $B = \gamma W \cap [w a_{2k+2}] \cap L$ with respect to $\~\mu$.
Thus, $\gamma^{-1} A = W \cap \gamma^{-1}[w a_1] \cap \gamma^{-1}L$ does not $k$-chain to $\gamma^{-1} B = W \cap \gamma^{-1}[w_{a 2k+2}] \cap \gamma^{-1}L$, so $\F_{2k+2} \actson (X,\~\mu)$ is not $k$-chaining on $W$ and we conclude that it is not essentially $k$-chaining.
\end{proof}

\def\MR#1{}
\bibliographystyle{amsalpha} 
\bibliography{ref}

\end{document}